\documentclass[
  a4paper,oneside,DIV=12,
  12pt,
  headsepline,
  ]{scrartcl}

\usepackage{etoolbox}
\usepackage{ifdraft}
\usepackage{iftex}

\ifluatex
\usepackage[utf8]{luainputenc}
\else
\usepackage[utf8]{inputenc}
\usepackage[T1]{fontenc}
\fi

\usepackage[english]{babel}
\usepackage[autostyle=true]{csquotes}

\usepackage{lmodern}
\usepackage{fourier}

\usepackage{amsfonts}
\usepackage{mathrsfs} %
\usepackage{dsfont}

\usepackage{amssymb}
\usepackage{stmaryrd}    %

\usepackage{amsmath}
\usepackage{mathtools}
\usepackage[thmmarks, amsmath, amsthm]{ntheorem}
\usepackage{nccmath}    %
\usepackage{tensor}

\usepackage[
  hscale=0.8,vscale=0.875,
  includehead,footskip=2\baselineskip,
  ]{geometry}

\usepackage[draft=false]{scrlayer-scrpage}
\usepackage{needspace}

\usepackage[cmyk,dvipsnames]{xcolor}

\ifdraft{%
\usepackage[textsize=scriptsize,colorinlistoftodos]{todonotes}
\usepackage{lineno}
\usepackage{scrtime}
}{}

\usepackage{booktabs}
\usepackage[inline]{enumitem}
\usepackage{lettrine}
\usepackage{url}

\usepackage[backend=biber,style=numeric-comp,giveninits,url=false,sortcites=true,maxbibnames=99]{biblatex}
\usepackage[final,
            pdfborder={0 0 0}, colorlinks=true,
            linkcolor=BrickRed, citecolor=ForestGreen, urlcolor=RoyalBlue]{hyperref}

\ifdraft{%
\linenumbers
\newcommand*\patchAmsMathEnvironmentForLineno[1]{%
  \expandafter\let\csname old#1\expandafter\endcsname\csname #1\endcsname
  \expandafter\let\csname oldend#1\expandafter\endcsname\csname end#1\endcsname
  \renewenvironment{#1}%
     {\linenomath\csname old#1\endcsname}%
     {\csname oldend#1\endcsname\endlinenomath}}%
\newcommand*\patchBothAmsMathEnvironmentsForLineno[1]{%
  \patchAmsMathEnvironmentForLineno{#1}%
  \patchAmsMathEnvironmentForLineno{#1*}}%
\AtBeginDocument{%
\patchBothAmsMathEnvironmentsForLineno{equation}%
\patchBothAmsMathEnvironmentsForLineno{align}%
\patchBothAmsMathEnvironmentsForLineno{flalign}%
\patchBothAmsMathEnvironmentsForLineno{alignat}%
\patchBothAmsMathEnvironmentsForLineno{gather}%
\patchBothAmsMathEnvironmentsForLineno{multline}%
}}

\clearmainofpairofpagestyles
\automark[section]{subsection}

\renewcommand{\subsectionmark}[1]{}

\lohead{{\small\normalfont \headertitle}}
\rohead{{\small \headerauthors}}

\RedeclareSectionCommand[%
  font=\Large\sffamily\bfseries,%
  beforeskip=1\baselineskip,%
  afterskip=0.5\baselineskip,%
  indent=0em,%
  afterindent=false,%
  tocbeforeskip=0.3\baselineskip plus 1pt minus 1pt%
  ]{section}

\RedeclareSectionCommands[%
  font=\normalfont\bfseries,%
  beforeskip=3pt,%
  afterskip=-1em%
  ]{subsection,subsubsection}

\RedeclareSectionCommands[%
  font=\normalfont\itshape,%
  beforeskip=.5\baselineskip,%
  afterskip=-1em,%
  indent=0pt,%
]{paragraph}

\AtEveryBibitem{\clearfield{doi}}
\AtEveryBibitem{\clearfield{isbn}}
\AtEveryBibitem{\clearfield{issn}}
\AtEveryBibitem{\clearfield{pages}}
\AtEveryBibitem{\clearlist{language}}

\renewcommand*{\bibfont}{\footnotesize}
\renewbibmacro*{in:}{}
\DeclareFieldFormat
  [article,inbook,incollection,inproceedings,patent,thesis,unpublished,misc]
  {title}{#1}

\setitemize{itemsep=0.02\baselineskip}

\newenvironment{enumeratearabic}{
\begin{enumerate}[label=(\arabic*),%
  leftmargin=2.5em,itemindent=0pt,%
  labelindent=.5em,labelwidth=1.5em,labelsep=!,%
  noitemsep]
}{
\end{enumerate}
}

\newenvironment{enumerateroman}{
\begin{enumerate}[label=(\roman*),%
  leftmargin=2.5em,itemindent=0pt,%
  labelindent=.5em,labelwidth=1.5em,labelsep=!,%
  nosep]
}{
\end{enumerate}
}

\newenvironment{enumeratearabic*}{
\begin{enumerate*}[label=(\arabic*)] %
}{
\end{enumerate*}
}

\newenvironment{enumerateroman*}{
\begin{enumerate*}[label=(\roman*)] %
}{
\end{enumerate*}
}

\numberwithin{equation}{section}

\theoremnumbering{arabic}
\newtheorem{theoremcounter}{theoremcounter}[section]
\theoremnumbering{Alph}
\newtheorem{maintheoremcounter}{maintheoremcounter}

\theoremstyle{plain}

\newtheorem{corollary}[theoremcounter]{Corollary}
\newtheorem{lemma}[theoremcounter]{Lemma}

\newtheorem{proposition}[theoremcounter]{Proposition}

\theoremstyle{plain}

\newtheorem{maintheorem}[maintheoremcounter]{Theorem}

\theoremstyle{definition}

\newtheorem{definition}[theoremcounter]{Definition}

\theoremstyle{remark}

\newtheorem{remark}[theoremcounter]{Remark}

\theoremstyle{nonumberremark}

\newtheorem{mainremark}{Remark}

\newenvironment{mainremarkenumerate}
{%
\mainremark
\enumeratearabic
}{%
\endenumeratearabic
\endmainremark
}%

\newcommand{\tx}{\text}
\newcommand{\thdash}{\nbd th}
\newcommand{\nbd}{\nobreakdash-\hspace{0pt}}

\newcommand{\fref}[2]{\hyperref[#2]{#1~\ref*{#2}}}

\newcommand{\texpdf}[2]{\texorpdfstring{#1}{#2}}

\makeatletter
\newcommand{\writelabel}[1]{#1\def\@currentlabel{#1}}
\makeatother

\newcommand{\minwidthmathbox}[2]{%
  \mathmakebox[{\ifdim#1<\width\width\else#1\fi}]{#2}%
}

\newcommand{\tbf}{\bfseries}

\newcommand{\bbM}{\mathbb{M}}

\newcommand{\bbone}{\mathds{1}}

\newcommand{\cO}{\mathcal{O}}

\newcommand{\frake}{\mathfrak{e}}

\newcommand{\rmd}{\mathrm{d}}

\newcommand{\rmt}{\mathrm{t}}

\newcommand{\rmv}{\mathrm{v}}

\newcommand{\rmH}{\mathrm{H}}

\newcommand{\rmM}{\mathrm{M}}

\newcommand{\rmS}{\mathrm{S}}

\newcommand{\td}{\tilde}

\newcommand{\ov}{\overline}

\newcommand{\defcol}{\mathrel{:}}
\newcommand{\defeq}{\mathrel{:=}}

\newcommand{\condsep}{\mathrel{\;:\;}}

\newcommand{\mrelspace}[1]{\mathrel{\mspace{#1}}}

\NewCommandCopy{\rightarroworig}{\rightarrow}
\renewcommand{\rightarrow}
  {\protect\relbar\mrelspace{-9.7mu}\rightarroworig}

\NewCommandCopy{\leftarroworig}{\leftarrow}
\renewcommand{\leftarrow}
  {\protect\leftarroworig\mrelspace{-9.7mu}\relbar}

\renewcommand{\longrightarrow}
  {\protect\relbar\mrelspace{-3.2mu}\relbar\mrelspace{-9.5mu}\rightarroworig}
\newcommand{\longhookrightarrow}
  {\protect\lhook\mrelspace{-3.1mu}\relbar\mrelspace{-3.2mu}\relbar\mrelspace{-11.7mu}\rightarroworig}

\newcommand{\ra}{\rightarrow}

\newcommand{\lra}{\longrightarrow}
\newcommand{\lhra}{\longhookrightarrow}

\newcommand{\mto}{\mapsto}
\newcommand{\lmto}{\longmapsto}

\renewcommand{\pmod}[1]{\;(\mathrm{mod}\, #1)}

\newcommand{\id}{\mathrm{id}}

\newenvironment{psmatrix}{\left(\begin{smallmatrix}}{\end{smallmatrix}\right)}

\newcommand{\lT}[1]{\mathord{\mspace{1mu}\tensor*[^\rmt]{#1}{}}}

\newcommand{\ZZ}{\mathbb{Z}}
\newcommand{\QQ}{\mathbb{Q}}
\newcommand{\RR}{\mathbb{R}}
\newcommand{\CC}{\mathbb{C}}

\newcommand{\GL}[1]{\mathrm{GL}_{#1}}
\newcommand{\Mp}[1]{\mathrm{Mp}_{#1}}

\newcommand{\SL}[1]{\mathrm{SL}_{#1}}

\newcommand{\U}[1]{\mathrm{U}_{#1}}

\newcommand{\HS}{\mathbb{H}}

\newcommand{\Ga}{\Gamma}
\newcommand{\ga}{\gamma}
\newcommand{\om}{\omega}

\newcommand{\hol}{\mathrm{hol}}
\DeclareMathOperator{\Ind}{Ind}
\DeclareMathOperator{\Res}{Res}
\newcommand{\genInd}{\operatorname{genInd}}

\newcommand{\intrmd}{\rmd\mspace{-2mu}}

\newcommand{\pluspad}{\mathbin{\mspace{2mu}+\mspace{2mu}}}

\newcommand{\vdotssmall}{\vphantom{\int\limits^{.}}\smash{\vdots}}

\newcommand{\rhoH}{\rho_{\rmH}}

\newcommand{\ovsmash}[1]{\smash{\ov{#1}}\vphantom{A}}

\title{%
  Weighted Recursions for\\Hurwitz Class Numbers
}
\author{%
  Matthew Ortiz%
  \and%
  Martin Raum%
  \thanks{The author was partially supported by Vetenskapsr\aa det Grant~2023-04217.}%
  \and%
  Olav K. Richter
  \thanks{The author was partially supported by a grant from the Simons Foundation (\#835652 to Olav Richter).}
}
\newcommand{\headertitle}{%
  Weighted Recursions for Hurwitz Class Numbers
}
\newcommand{\headerauthors}{%
  M.~Ortiz,
  M.~Raum,
  O.~Richter
}
\ifdraft{\date{\today\ at\ \thistime}}{\date{}}

\begin{document}

\thispagestyle{scrplain}
\begingroup
\deffootnote[1em]{1.5em}{1em}{\thefootnotemark}
\maketitle
\endgroup

\begin{abstract}
  \small
  \noindent
  {\tbf Abstract:}
  We establish new recursions for Hurwitz class numbers with polynomial weights.
  In contrast to previous recursions, our results decouple class numbers of even and odd discriminants.
  Our main tool is the vector-valued holomorphic projection operator applied to mock modular forms.
  We invoke representation theory to connect the relevant spaces of vector-valued modular forms to spaces of classical new and old forms.
  We thereby leverage the vanishing of spaces of vector-valued cusp forms not available in the scalar case.
  \\[.3\baselineskip]
  \noindent
  \textsf{\textbf{%
      Hurwitz class numbers%
    }}%
  \noindent
  \ {\tiny$\blacksquare$}\ %
  \textsf{\textbf{%
      vector-valued mock modular forms%
    }}%
  \noindent
  \ {\tiny$\blacksquare$}\ %
  \textsf{\textbf{%
      holomorphic projection%
    }}%
  \nolinebreak
  \ {\tiny$\blacksquare$}\ %
  \noindent
  \textsf{\textbf{%
      Rankin--Cohen brackets%
    }}
  \\[.2\baselineskip]
  \noindent
  \textsf{\textbf{%
      MSC Primary: 11E41%
    }}%
  \ {\tiny$\blacksquare$}\ %
  \textsf{\textbf{%
      MSC Secondary: 11F30, 11F37%
    }}
\end{abstract}

\Needspace*{4em}
\phantomsection
\label{sec:introduction}
\addcontentsline{toc}{section}{Introduction}
\markright{Introduction}

\lettrine[lines=2,nindent=.2em]{\tbf I}{}ntroduced over a century ago, {Hurwitz} class numbers~$H(n)$ encode fundamental arithmetic information: For a negative fundamental discriminant~$-D < -4$,~$H(D)$ equals the class number of the imaginary quadratic field~$\QQ(\sqrt{-D})$, weighted by the reciprocal of the number of units modulo~$\pm 1$. Recursions for~$H(n)$ have a particularly rich history, dating back to the work of Kronecker, Gierster, and Hurwitz in the 19th century~\cite{kronecker-1860,giester-1882,hurwitz-1885}. With~$\sigma_s(N) = \sum_{ab = N} a^s$ and~$\lambda_s(N) = \sum_{ab = N} \min\{a,b\}^s$, a classical example is the relation
\begin{alignat}{5}
  \label{eq:classical_hurwitz_class_number_recursion}
  \tag{HR}
   &
  \sum_{\mathclap{m \in \ZZ}}
  \mspace{3mu}
  H(4N - m^2)
   &   & =
  -
  \lambda_1(N)
  +
  2 \sigma_1(N)
  \quad
   &   & \tx{for all } N \in \ZZ_{>0}
  \tx{.}
   &
  \intertext{Later, Cohen~\cite{cohen-1975} exhibited weighted recursions for Hurwitz class numbers, for instance}
  \label{eq:hurwitz_class_number_recursion_Cohen1}
  \tag{wHR1}
   &
  \sum_{\mathclap{m \in \ZZ}}
  \mspace{3mu}
  (m^2-N)H(4N-m^2)
   &   & =
  -\lambda_3(N)
  \quad
   &   & \tx{for all } N \in \ZZ_{>0}
  \tx{;}
  \\
  \label{eq:hurwitz_class_number_recursion_Cohen2}
  \tag{wHR2}
   &
  \sum_{\mathclap{m \in \ZZ}}
  \mspace{3mu}
  (m^4-3Nm^2+N^2)
  H(4N-m^2)
   &   & =
  -\lambda_5(N)
  \quad
   &   & \tx{for all } N \in \ZZ_{>0}
  \tx{.}
\end{alignat}
Cohen conjectured further such identities, which were established by Mertens~\cite{mertens-2014}.

A common feature of the classical recursion in~\eqref{eq:classical_hurwitz_class_number_recursion} and many other known weighted recursions of Hurwitz class numbers is that they mix class numbers for both even and odd discriminants. One exception to this is the work of~\cite{williams-2019}, which gave unweighted recursions for~$H(n)$ separated by parity. Our main \fref{Theorem}{mainthm:hurwitz_class_number_recursions} provides weighted recursions for Hurwitz class numbers that likewise feature only even or only odd discriminants.

\begin{maintheorem}%
\label{mainthm:hurwitz_class_number_recursions}
Let $N \in \ZZ_{>0}$. The Hurwitz class numbers satisfy the following recursions:
\begin{alignat*}{4}
   &
  \sum_{\mathclap{m^2 \leq N}}
  \mspace{24mu}
   &   &
  g_3^{\mathrm{even}}(N, m)\,
   &   &
  H(4N - 4m^2)
   &   & =
  \lambda^{\mathrm{even}}_3(N)
  \tx{,}
  \\
   &
  \sum_{\mathclap{m^2 + m + 1 \leq N}}
   &   &
  g_3^{\mathrm{odd}}(N, m)\,
   &   &
  H(4N - 4m^2 - 4m - 1)
   &   & =
  \lambda^{\mathrm{odd}}_3(N)
  \tx{,}
  \\
   &
  \sum_{\mathclap{m^2 \leq 2N}}
   &   &
  g_5^{\mathrm{even}}(2N, m)\,
   &   &
  H(8N - 4m^2)
   &   & =
  \lambda^{\mathrm{even}}_5(2N)
  \tx{,}
  \\
   &
  \sum_{\mathclap{m^2 + m + 1 \leq 2N}}
   &   &
  g_5^{\mathrm{odd}}(2N, m)\,
   &   &
  H(8N - 4m^2 - 4m - 1)
   &   & =
  \lambda^{\mathrm{odd}}_5(2N)
  \tx{,}
\end{alignat*}
where
\begin{align*}
  g_3^{\mathrm{even}}(N, m)
   & =
  N - 4m^2
  \tx{,}\mspace{8mu}
  g_3^{\mathrm{odd}}(N, m)
  =
  N - 4m^2 - 4m - 1
  \tx{,}\mspace{8mu}
  g_5^{\mathrm{even}}(N, m)
  =
  -16m^4 + 12Nm^2 - N^2
  \tx{,}
  \\
  g_5^{\mathrm{odd}}(N, m)
   & =
  -16m^4 - 32m^3 + 12(N-2)m^2 + 4(3N-2)m - N^2 + 3N - 1
\end{align*}
and
\begin{gather}
  \label{eq:holomorphic_projection:hurwitz:lambdas}
  \tag{\text{$\Lambda$}}
  \lambda_k^{\mathrm{even}}(N)
  \defeq
  \mspace{-12mu}
  \sum_{\substack {ab=N \\ a\equiv b\pmod{2}}}
  \mspace{-9mu}
  \min\{a,b\}^k
  \tx{,}\qquad
  \lambda_k^{\mathrm{odd}}(N)
  \defeq
  \mspace{-12mu}
  \sum_{\substack {ab=N \\ a\not\equiv b\pmod{2}}}
  \mspace{-9mu}
  \min\{a,b\}^k
  \tx{.}
\end{gather}
\end{maintheorem}

\fref{Theorem}{mainthm:hurwitz_class_number_recursions} is enabled by a vector-valued approach, while previous proofs of recursions of Hurwitz class numbers largely relied on scalar-valued forms.  A key advantage of the vector-valued setting is the availability of further vanishing results for spaces of cusp forms.

\begin{mainremarkenumerate}
\item
The functions in~\fref{Corollary}{cor:representation_theory:hurwitz_modular_forms_support}, \fref{Proposition}{prop:holomorphic_projection:hurwitz:holomorphic_part}, and \fref{Proposition}{prop:holomorphic_projection:hurwitz:nonholomorphic_part} take values in a~$4$-dimensional~$\CC$-vector space. \fref{Theorem}{mainthm:hurwitz_class_number_recursions} arises from inspecting their first and fourth components. Moreover, considering the sum of the first and fourth components allows one to recover Cohen's~\cite{cohen-1975} recursions~\eqref{eq:hurwitz_class_number_recursion_Cohen1} and~\eqref{eq:hurwitz_class_number_recursion_Cohen2}.  Finally, investigating the second and third components yields the formulas in Corollary~2 of~\cite{mertens-2014}. Observe that for weights larger than~$4$, these components no longer vanish, which explains the appearance of the theta function coefficients in~\cite{mertens-2014}.
\item
  A vector-valued setup was employed in~\cite{imamoglu-raum-richter-2014} to establish recursions for Fourier coefficients of Ramanujan's third order mock theta functions.
  After this paper was written, in~\cite{ortiz-raum-richter-2026}, we iterated on the vector-valued Rankin--Cohen setup presented here to obtain further recursions for them.
  Similar methods have also subsequently been used by~\cite{bradley-thrush-garvan-kalita-rolen-2026} to obtain convolution identities for these and other mock theta functions.
\item
One interpretation of our recursions is that we restrict~$m \,\pmod{2}$. This is similar to the recursions in~\cite{bringmann-kane-2019}, where~$m$ was restricted modulo odd primes.
\end{mainremarkenumerate}

Besides our vector-valued framework, our main ingredients are holomorphic projection and Rankin--Cohen brackets. We apply these tools to the tensor product of a vector-valued version of Zagier's Eisenstein series and the vector-valued version of the classical Jacobi theta series. To describe their modular behavior we use a representation~$\rhoH$ of the metaplectic group~$\Mp{2}(\ZZ)$. A crucial structural insight is the decomposition of the tensor product representation~$\ov{\rhoH} \otimes \rhoH$ into irreducible components, which allows us to identify projected forms in explicit spaces of modular forms.  In particular, our formulas depend critically on the vanishing of certain spaces of vector-valued cusp forms.

This paper is organized as follows. We introduce vector-valued modular and Maass forms, Rankin--Cohen brackets, and holomorphic projection in \fref{Section}{sec:preliminaries}.
In \fref{Section}{sec:vector_valued_modular_forms}, we describe the spaces of vector-valued modular forms we work with, and derive the vanishing of the subspaces of cusp forms.
We calculate the Fourier coefficients that appear in the holomorphic projections we use in \fref{Section}{sec:holomorphic_projection_fourier_coefficients}.
Finally, we combine the results of \fref{Section}{sec:vector_valued_modular_forms} and \fref{Section}{sec:holomorphic_projection_fourier_coefficients} to prove \fref{Theorem}{mainthm:hurwitz_class_number_recursions} in \fref{Section}{sec:proof_main_theorem}.

\paragraph{Acknowledgements} The authors thank the anonymous referee for a thorough reading of the manuscript and for constructive remarks that improved this paper.

\section{Preliminaries}%
\label{sec:preliminaries}
\subsection{The metaplectic modular group}
\label{ssec:preliminaries:metaplectic}

We introduce notation for vector-valued modular forms using standard conventions. For~$z \in \CC$, we set~$e(z) \defeq \exp(2 \pi i\, z)$. We set~$\zeta_r \defeq e(1 / r)$, and denote by~$\HS \defeq \{ \tau = x + i y \in \CC \condsep y > 0 \}$ the Poincaré upper half-plane. For any periodic function~$f$ on~$\HS$, we write~$c(f;\, n; y)$ for its~$n$\thdash{} Fourier coefficient,~$n \in \QQ$. When this coefficient does not depend on~$y$, we use the simplified notation~$c(f;\, n)$.

We denote by~$\Ga_0(N)$ the standard congruence subgroups of~$\SL{2}(\ZZ)$.

The metaplectic group~$\Mp{2}(\RR)$ is a double cover of~$\SL{2}(\RR)$ consisting of pairs~$(g, \om)$ where~$g = \begin{psmatrix} a & b \\ c & d \end{psmatrix} \in \SL{2}(\RR)$ and~$\om \defcol \HS \ra \CC$ is a holomorphic function satisfying~$\om(\tau)^2 = c \tau + d$. The multiplication rule is given by
\begin{gather*}
  (g_1, \omega_1)\, (g_2, \omega_2)
  =
  \big( g_1 g_2,\, (\omega_1 \circ g_2) \cdot \omega_2 \big)
  \tx{.}
\end{gather*}
We call the preimage of~$\SL{2}(\ZZ)$ under the projection~$\Mp{2}(\RR) \ra \SL{2}(\RR)$ the metaplectic modular group~$\Mp{2}(\ZZ)$. It is generated by
\begin{gather*}
  S
  \defeq
  \big( \begin{psmatrix} 0 & -1 \\ 1 & 0 \end{psmatrix},\, \sqrt{\tau} \big)
  \quad\text{and}\quad
  T
  \defeq
  \big( \begin{psmatrix} 1 & 1 \\ 0 & 1 \end{psmatrix},\, 1 \big)
  \tx{,}
\end{gather*}
where the square root denotes the principal branch satisfying~$\sqrt{i} = \zeta_8$.

\subsection{Modular forms and harmonic Maass forms}%
\label{ssec:preliminaries:harmonic_maass_forms}

We recall scalar- and vector-valued modular forms and harmonic Maass forms~\cite{bringmann-folsom-ono-rolen-2018}.

\paragraph{Classical modular forms and Maass forms}
For a subgroup~$\Ga \subseteq \SL{2}(\ZZ)$, we write~$\rmM_k(\Ga)$ for the space of modular forms of weight~$k$, which may be integral or half-integral, and~$\bbM_k(\Ga)$ for the space of harmonic Maass forms of that weight that are of moderate growth at all cusps. For spaces of modular forms and harmonic Maass forms for characters, we often omit that character and write~$\rmM_k(\Ga)$ and~$\bbM_k(\Ga)$ instead of~$\rmM_k(\Ga,\chi)$ and~$\bbM_k(\Ga,\chi)$. We also work with quasi-modular forms of integral weight~\cite{zagier-1994, kaneko-zagier-1995}.

\paragraph{Vector-valued modular forms}
We consider finite dimensional unitary representations~$\rho$ of the metaplectic group~$\Mp{2}(\ZZ)$, which we call (arithmetic) types. 
Since we work with Fourier expansions with rational Fourier indices, we further suppose~$\rho(T)$ has finite order.
We denote by~$V(\rho)$ the underlying vector space carrying the representation~$\rho$, equipped with an inner product~$\langle \cdot\, ,\,\cdot \rangle_\rho$ that makes~$\rho$ unitary. When~$\rho(S)^2 = \id$ is the identity on~$V(\rho)$, the representation~$\rho$ descends to a representation of~$\SL{2}(\ZZ)$.

Given a half-integer weight~$k$, we define an action on functions~$f \defcol \HS \ra V(\rho)$ via the weight~$k$ slash operator of type~$\rho$. For any element~$(\ga,\om) \in \Mp{2}(\ZZ)$,~$\ga = \begin{psmatrix} a & b \\ c & d \end{psmatrix}$, this operator acts as:
\begin{gather*}
  \big( f \big|_{k, \rho}\, \ga \big)(\tau)
  \defeq
  \rho\big((\ga,\om)\big)^{-1}\,
  \om(\tau)^{-2k}\,
  f\big( \mfrac{a \tau + b}{c \tau + d} \big)
  \tx{.}
\end{gather*}

Let~$\rmM_{k}(\rho)$ and~$\rmS_{k}(\rho)$ be the spaces of modular forms and cusp forms of weight~$k$ and type~$\rho$, comprising those holomorphic functions~$\HS \rightarrow V(\rho)$ that remain invariant under the~$|_{k, \rho}$ action and satisfy appropriate growth conditions at the cusps. Likewise, $\bbM_{k}(\rho)$ denotes the corresponding space of vector-valued Maass forms.

In analogy with scalar-valued modular forms, a vector-valued modular form~$f \in \rmM_k(\rho)$ being bounded at infinity implies that its Fourier series expansion takes the form
\begin{gather*}
  f
  =
  \sum_{n \in \QQ}
  c(f;n)\, e(n \tau)
  \tx{,}
\end{gather*}
where the coefficients~$c(f;n)$ are elements of~$V(\rho)$.

\paragraph{Eichler integrals}
Let~$f(\tau) = \sum_{n \ge 0} c(f;n)\, e(n \tau)$ be a holomorphic modular form in~$\rmM_{2-k}(\Ga)$ or in~$\rmM_{2-k}(\rho)$, where~$k > 1$, $\Ga$ is a congruence subgroup, and~$\rho$ is a type. Its non-holomorphic Eichler integral, denoted by~$f^\ast(\tau)$, is defined by
\begin{gather}
  \label{eq:def:non-holomorphic-eichler-integral}
  \begin{aligned}
  f^\ast(\tau)
   & :=
  -(2 i)^{k-1}
  \int_{-\ov{\tau}}^{i\infty}
  \frac{\ov{f(-\ov{w})}}{(w + \tau)^k}
  \,\rmd\mspace{-2mu}w
  \\
   & \hphantom{:}=
  \frac{\ov{c(f,0)}}{1-k}\, y^{1-k}
  \,-\,
  (4 \pi)^{k-1}
  \sum_{n < 0}
  \ov{c(f, |n|)\, }|n|^{k-1}\,
  \Gamma(1-k,4 \pi |n| y) e(n\tau)
  \tx{,}
  \end{aligned}
\end{gather}
where in the formula~$\Ga$ here is the upper incomplete Gamma-function.

\subsection{Class number generating series}%
\label{ssec:preliminaries:class_number_generating_series}

Recall the theta function
\begin{gather}
  \label{eq:def:theta}
  \theta(\tau)
  \defeq
  \sum_{m \in \ZZ}
  e\big( m^2 \tau \big)
  \,\in\,
  \rmM_{\frac{1}{2}}(\Ga_0(4))
  \tx{.}
\end{gather}
Recall that~$H(D) = 0$ if~$D < 0$ or~$-D$ is \emph{not} a discriminant.
Zagier~\cite{zagier-1975} showed that the holomorphic generating series~$\sum_D H(D) e(D \tau)$ of Hurwitz class numbers admits the following modular completion:
\begin{gather}
  \label{eq:zagier-eisenstein-series}
  E_{\frac{3}{2}}
  \defeq
  E^+_{\frac{3}{2}}
  +
  E^-_{\frac{3}{2}}
  \,\in\,
  \bbM_{\frac{3}{2}}(\Ga_0(4))
  \tx{,}\quad
  E^-_{\frac{3}{2}}
  \defeq
  - \mfrac{1}{16 \pi} \theta^\ast
  \tx{,}
\end{gather}
where
\begin{gather*}
  E^+_{\frac{3}{2}}(\tau)
  \defeq
  \sum_{D = 0}^\infty H(D)\, e(D \tau)
  \tx{,}\quad
  E^-_{\frac{3}{2}}(\tau)
  =
  \mfrac{1}{8 \pi} y^{-\frac{1}{2}}
  +
  \mfrac{1}{8 \sqrt{\pi}}
  \sum_{m \in \ZZ \setminus \{0\}}
  \mspace{-6mu}
  |m|\,
  \Gamma\big( -\tfrac{1}{2}, 4 \pi m^2 y \big)\,
  e\big( -m^2 \tau \big)
  \tx{.}
\end{gather*}

There are vector-valued analogs (see~\cite{williams-2019b}) that transform with the representation~$\rhoH$ of~$\Mp{2}(\ZZ)$ given by
\begin{gather}
  \label{eq:def:representation:hurwitz}
  \rhoH(S)
  =
  \mfrac{1+e(\frac{3}{4})}{2}
  \begin{psmatrix}
    1 & 1 \\
    1 & -1
  \end{psmatrix}
  \tx{,}\quad
  \rhoH(T)
  =
  \begin{psmatrix}
    1 &  \\
    & e(\frac{1}{4})
  \end{psmatrix}
  \tx{.}
\end{gather}
These analogs are given by
\begin{gather*}
  \theta^{\rmv\rmv}(\tau)
  =
  \lT{\Big(
    \sum_{m \in 2\ZZ}
    e\big( \mfrac{m^2}{4} \tau \big)
    \tx{,}\,
    \sum_{m \in 1 + 2\ZZ}
    e\big( \mfrac{m^2}{4} \tau \big)
    \Big)}
  \,\in\,
  \rmM_{\frac{1}{2}}(\rhoH)
\end{gather*}
and
\begin{gather*}
  E^{\rmv\rmv}_{\frac{3}{2}}
  \defeq
  E^{\rmv\rmv\,+}_{\frac{3}{2}}
  +
  E^{\rmv\rmv\,-}_{\frac{3}{2}}
  \,\in\,
  \bbM_{\frac{3}{2}}(\ov\rhoH)
  \tx{,}\quad
  E^{\rmv\rmv\,-}_{\frac{3}{2}}
  \defeq
  - \mfrac{1}{8 \pi}
  \theta^{\rmv\rmv\,\ast}
  \tx{,}
\end{gather*}
where
\begin{align*}
  E^{\rmv\rmv\,+}_{\frac{3}{2}}
   & \defeq
  \lT{\Big(}
  \sum_{n \in 4 \ZZ}
  H(n)\, e\big( \mfrac{n}{4} \tau \big)
  ,\;
  \sum_{n \in 3 + 4\ZZ}
  \mspace{-6mu}
  H(n)\, e\big( \mfrac{n}{4} \tau \big)
  \Big)
  \tx{,}
  \\
  E^{\rmv\rmv\,-}_{\frac{3}{2}}
   & \hphantom{:}=
  \lT{\Big(}
  \mfrac{1}{4 \pi} y^{-\frac{1}{2}}
  +
  \mfrac{1}{8 \sqrt{\pi}}
  \sum_{m \in 2\ZZ \setminus \{0\}}
  \mspace{-6mu}
  |m|\,
  \Gamma\big( -\tfrac{1}{2}, \pi m^2 y \big)\,
  e\big( -\tfrac{m^2}{4} \tau \big)
  ,\;
  \mfrac{1}{8 \sqrt{\pi}}
  \sum_{m \in 1 + 2\ZZ}
  \mspace{-6mu}
  |m|\,
  \Gamma\big( -\tfrac{1}{2}, \pi m^2 y \big)\,
  e\big( -\tfrac{m^2}{4} \tau \big)
  \Big)
  \tx{.}
\end{align*}

Note the difference in the normalization between the vector-valued and scalar-valued case.
The extra factor of 2 accounts for the lack of equivariance for the Hecke operator in~$\GL{2}(\RR)$, which results in additional constants appearing in the ``vectorization"~$E_{\frac{3}{2}}^{\rmv\rmv-}$ of~$E_{\frac{3}{2}}^{-}$.
We note that a qualitatively similar phenomenon occurs in Equations~\eqref{eq:representation_theory:hurwitz:ind_ga04_modular_forms_wt6} and~\eqref{eq:representation_theory:hurwitz:ind_ga04_modular_forms_wt8} below, where we also ``vectorize" scalar-valued modular forms.

\subsection{Rankin--Cohen brackets}%
\label{ssec:preliminaries:rankin_cohen}

Given two modular forms, or a modular form and a harmonic weak Maass form, one can construct new modular forms from them using the Rankin--Cohen bracket:
\begin{definition}
\label{def:rankin-cohen}
Let~$f,g \defcol \HS \rightarrow \CC$ be real-analytic of weights~$k$ and~$\ell$, respectively. The \emph{$\nu$\thdash{} Rankin--Cohen bracket} of~$f$ and~$g$ is defined by
\begin{gather}
  \label{eq:def:rankin-cohen}
  \big[ f,g \big]_\nu
  \defeq
  \sum_{\mu=0}^\nu(-1)^\mu
  \mbinom{k+\nu-1}{\nu-\mu}
  \mbinom{l+\nu-1}{\mu}\,
  f^{(\mu)}\,
  g^{(\nu-\mu)}
  \tx{,}
\end{gather}
where~$f^{(\mu)} = \frac{1}{(2\pi i)^\mu}\partial_\tau^{\mu}f$ is the~$\mu$\thdash{} normalized derivative of~$f$ (and similarly for~$g$).
\end{definition}
Note that in particular,~$[f,g]_0 = fg$. As established by Rankin and Cohen~\cite{rankin-1956, cohen-1975}, if~$f$ and~$g$ are modular forms of weight~$k$ and~$\ell$, then~$[f,g]_\nu$ is a modular form of weight~$k+\ell+2\nu$. Furthermore, by Corollary~7.2 of~\cite{cohen-1975},~$[f,g]_\nu$ is a cusp form for~$\nu \geq 1$.

\subsection{Holomorphic projection}%
\label{ssec:preliminaries:holomorphic_projection}

We recall the following definition from~\cite{imamoglu-raum-richter-2014}:

\begin{definition}
\label{def:holomorphic_projection}
Fix an integer~$k\ge 2$.
Let~$V$ be a finite-dimensional~$\CC$-vector space, and suppose that~$f \defcol \HS \rightarrow V$ is continuous with Fourier expansion
\begin{gather*}
  f(\tau)
  =
  \sum_{m \in \frac{1}{N}\ZZ}
  c(f;m;y)\, e(m\tau)
\end{gather*}
for some~$N \in \ZZ_{> 0}$. Furthermore, assume that
\begin{enumerateroman}
\item
\label{it:def:holomorphic_projection:constant}
$f(\tau) = c_0 + \cO(y^{-\epsilon})$ as~$y \rightarrow \infty$ for some~$\epsilon > 0$ and~$c_0 \in V$, and
\item
\label{it:def:holomorphic_projection:growth}
for all~$m > 0$, $c(f; m; y) = \cO(y^{1-k+\epsilon})$ as~$y \rightarrow 0$ for some~$\epsilon > 0$.
\end{enumerateroman}
Then the holomorphic projection of~$f$ is defined by
\begin{gather}
  \label{eq:def:holomorphic_projection}
  \pi^\hol_{k}(f)
  \defeq
  c_0 + \sum_{0 < m \in \frac{1}{N}\ZZ}
  c(m)\, e(m \tau)
  \quad\tx{with}\quad
  c(m)
  \defeq
  \mfrac{(4\pi m)^{k-1}}{\Gamma(k-1)}\,
  \lim_{s \ra 0}\,
  \int\limits_0^\infty c(f; m; y)\, e^{-4\pi m y} y^{s+k-2} \,\intrmd y
  \tx{,}
\end{gather}
where~$\Ga$ denotes the Gamma function.
\end{definition}
\fref{Assumptions}{it:def:holomorphic_projection:constant} and~\ref{it:def:holomorphic_projection:growth} ensure convergence of the integral in~\eqref{eq:def:holomorphic_projection}. Furthermore, Proposition~4 of~\cite{imamoglu-raum-richter-2014} shows that~$\pi^\hol_k(f) = f$ for any holomorphic function~$f$. Theorem~5 of the same reference, whose proof extends verbatim to the slightly weaker \fref{Assumption}{it:def:holomorphic_projection:growth} compared to what is assumed in~\cite{imamoglu-raum-richter-2014}, shows that~$\pi^\hol_k(F)$ is a modular form provided~$k>2$ and~$F$ is invariant under the slash operator and satisfies a certain growth condition. In particular, this holds for the Rankin--Cohen brackets we will consider, where we will have~$k\ge 4$.

\section{Vector-valued modular forms}%
\label{sec:vector_valued_modular_forms}

Throughout, we work with matrix representations; that is, we fix a basis for each representation space we consider.

\subsection{Induced representations and modular forms}%
\label{ssec:vector_valued_modular_forms:induced_representations}

To define inductions, it is more convenient to view a representation~$\rho$ of a congruence subgroup~$\Ga$ as a right~$\CC[\Ga]$-module.
In an abuse of notation, we write~$\rho$ to stand for this~$\CC[\Ga]$-module.
Furthermore, to enable explicit calculations in \fref{Section}{ssec:vector_valued_modular_forms:induced_representation_ga04}, we will usually fix a basis of the underlying representation space and view all representations as matrix representations. This also allows us to identify vector-valued modular forms with vectors of modular forms.

\begin{definition}
\label{def:induced_representation}
Let~$\Ga\subseteq\SL{2}(\ZZ)$ be a congruence subgroup and~$\rho$ a representation of~$\Ga$. The~\emph{induced representation} is the representation of~$\SL{2}(\ZZ)$ given by
\begin{gather*}
  \Ind_{\Ga}^{\SL{2}(\ZZ)}\,\rho
  \defeq
  \rho
  \otimes_{\CC[\Ga]}
  \CC[\SL{2}(\ZZ)]
  \tx{.}
\end{gather*}
\end{definition}
The tensor product encodes the fact that we do not need to consider all matrices in~$\SL{2}(\ZZ)$, but only those in different equivalence classes modulo~$\Ga$. Since~$\Ga$ is a congruence subgroup, there are finitely many such classes, and hence
$\Ind_{\Ga}^{\SL{2}(\ZZ)}\,\rho$
is finite-dimensional if~$\rho$ is. In particular, if~$\rho$ is a character, then
$\Ind_{\Ga}^{\SL{2}(\ZZ)}\,\rho$
is a representation of dimension~$[\SL{2}(\ZZ):\Ga]$, and we can choose a basis~$\frake_\ga$ for a set of representatives~$\ga$ of~$\Ga \backslash \SL{2}(\ZZ)$.

Vector-valued modular forms corresponding to a representation induced by a character are closely related to scalar-valued modular forms for that character.
More precisely, suppose~$f\in \rmM_k(\Ga,\chi)$ for a congruence subgroup~$\Ga\subseteq\SL{2}(\ZZ)$ of index~$n$ and a character~$\chi$ with finite index kernel.
Let~$\ga_1,\ldots,\ga_n\in\SL{2}(\ZZ)$ be representatives for the right cosets of~$\Ga$ in~$\SL{2}(\ZZ)$, with~$\ga_1$ being the identity. The order of these representatives allows us to identify functions
\begin{gather*}
  \tau \mto \sum_{i=1}^n f_i(\tau)\, \frake_{\ga_i}
  \quad\tx{with}\quad
  \begin{psmatrix}
    f_1 \\
    \vdotssmall \\
    f_n
  \end{psmatrix}
  \tx{.}
\end{gather*}

\begin{proposition}
\label{prop:vector_valued_modular_forms:induced_representations}
We have an isomorphism
\begin{gather*}
  \Ind
  \defcol
  \rmM_k(\Ga,\chi)
  \longrightarrow
  \rmM_k\bigl( \Ind_{\Ga}^{\SL{2}(\ZZ)}\chi \bigr)
  \tx{,}
  \quad
  f
  \lmto
  \begin{psmatrix}
    f|_k \ga_1  \\
    \vdotssmall \\
    f|_k \ga_n
  \end{psmatrix}
  \tx{.}
\end{gather*}
\end{proposition}

\begin{remark}
Without loss of generality, we choose~$\ga_1$ the identity, so that the inverse of~$\Ind$ is given by projecting onto the first component.
\end{remark}

\begin{proof}
This is analogous to Proposition~1.5 of~\cite{raum-2017}, but for convenience of the reader, we recall the details:
It is easy to see that~$\Ind$ is a linear map since~$|_k$ is.
Furthermore, $\Ind(f)$ is bounded at infinity because~$f$ is bounded at all of its cusps.
The modularity of~$\Ind(f)$ follows from the definition of induction. For ease of notation, we set~$\frake_i = \frake_{\ga_i}$. More precisely,
\begin{align*}
      &
  \Ind(f) \big|_k\, \ga
  =
  \Big(
  \sum_{i = 1}^n
  \big(f|_k\, \ga_i \big)\,
  \frake_i
  \Big)
  \Big|_k\, \ga
  =
  \sum_{i = 1}^n
  \big(f|_k\, \ga_i \ga \big)\,
  \frake_i
  =
  \sum_{i = 1}^n
  \big(f|_k\, \ga_i'\ga_{\sigma(i)} \big)\,
  \frake_i
  =
  \sum_{i = 1}^n
  \big(f|_k\, \ga_{\sigma(i)} \big)\,
  \chi(\ga_i')
  \frake_i
  \\
  ={} &
  \sum_{i = 1}^n
  \big(f|_k\, \ga_i \big)\,
  \chi(\ga_{\sigma^{-1}(i)}')
  \frake_{\sigma^{-1}(i)}
  =
  \sum_{i = 1}^n
  \big(f|_k\, \ga_i \big)\,
  \rho(\ga) \frake_i
  =
  \rho(\ga)\,
  \Ind(f)
\end{align*}
for a permutation~$\sigma \in \Sigma_n$ and~$\ga_i' = \ga_i \ga \ga_{\sigma(i)}^{-1}\in\Ga$.
Hence~$\Ind$ is a vector space homomorphism from~$\rmM_k(\Ga,\chi)$
to~$\rmM_k\bigl( \Ind_{\Ga}^{\SL{2}(\ZZ)}\chi \bigr)$.
Finally, the inverse of~$\Ind$ is given by the projection to the first component, confirming that~$\Ind$ is an isomorphism.
\end{proof}

The proof of our main theorem features a quotient of an induced representation, which discards contributions from groups larger than~$\Ga$. Specifically, note that if~$\chi$ extends to a character~$\chi'$ of~$\Ga' \subseteq \SL{2}(\ZZ)$ that properly contains~$\Ga$, by induction in steps we have a natural, proper inclusion
\begin{gather}
  \label{eq:representation_theory:hurwitz:ind_inclusion}
  \iota_{\Ga',\chi'}
  \defcol
  \Ind_{\Ga'}^{\SL{2}(\ZZ)}\, \chi'
  \lhra
  \Ind_{\Ga}^{\SL{2}(\ZZ)}\, \chi
  \tx{.}
\end{gather}
Since~$\Ga$ has finite index in~$\SL{2}(\ZZ)$ there are finitely many such~$\chi'$.

\begin{definition}
\label{def:genuine_induction}
Let~$\Ga \subseteq \SL{2}(\ZZ)$ be a congruence subgroup and~$\chi$ a character on~$\Ga$.
  Let~$\Xi$ be a fixed set of characters~$\chi'$, each on a subgroup~$\Ga \subsetneq \Ga' \subseteq \SL{2}(\ZZ)$, such that~$\Res_{\Ga}^{\Ga'}\chi'=\chi$.
The~\emph{genuine induction} (or~\emph{genuine quotient} of the induction) of~$\chi$ relative to~$\Xi$ with respect to~$\Ga$ is
\begin{gather}
  \label{eq:def:genuine_induction}
  \genInd_{\Ga}^{\SL{2}(\ZZ)}\, (\chi, \Xi)
  \defeq
  \Ind_{\Ga}^{\SL{2}(\ZZ)}\, \chi
  \mathbin{\Bigg\slash}
  \operatorname*{span}_{\chi' \in \Xi}
  \iota_{\Ga',\chi'}\Bigl(
  \Ind_{\Ga'}^{\SL{2}(\ZZ)}\, \chi'
  \Bigr)
  \tx{,}
\end{gather}
where the span on the right hand side is the $\CC$\nbd{}vector subspace generated by the images under the natural inclusions in~\eqref{eq:representation_theory:hurwitz:ind_inclusion}.
Below, we will work in the case where~$\chi=\bbone$ and $\Xi$ comprises the trivial characters on every proper supergroup of~$\Ga$, and we will simply write
  \begin{gather*}
    \genInd_\Ga^{\SL{2}(\ZZ)}(\bbone,\Xi)=\genInd_\Ga^{\SL{2}(\ZZ)}\,\bbone
  \end{gather*}
in this situation.
\end{definition}

\begin{remark}
The sum on the right hand side of~\eqref{eq:def:genuine_induction} is not a direct sum in general. For instance, if~$\Ga \subsetneq \Ga' \subsetneq \Ga''$ are subgroups,~$\chi''$ extends~$\chi$ to~$\Ga''$, and~$\chi'$ is its restriction to~$\Ga'$, then
\begin{gather*}
  \Ind_{\Ga''}^{\SL{2}(\ZZ)}\, \chi''
  \subsetneq
  \Ind_{\Ga'}^{\SL{2}(\ZZ)}\, \chi'
  \tx{,}
  \quad\tx{i.e., }\quad
  \Ind_{\Ga''}^{\SL{2}(\ZZ)}\, \chi''
  +
  \Ind_{\Ga'}^{\SL{2}(\ZZ)}\, \chi'
  =
  \Ind_{\Ga'}^{\SL{2}(\ZZ)}\, \chi'
  \tx{.}
\end{gather*}
\end{remark}

An analog of \fref{Proposition}{prop:vector_valued_modular_forms:induced_representations}
holds for the genuine induction, allowing us to describe the vector-valued modular forms arising from it.

The inclusion in~\eqref{eq:representation_theory:hurwitz:ind_inclusion} induces an injective map of modular forms via composition:
\begin{gather*}
  \rmM_k\big( \Ind_{\Ga'}^{\SL{2}(\ZZ)}\, \chi' \big)
  \lhra
  \rmM_k\big( \Ind_{\Ga}^{\SL{2}(\ZZ)}\, \chi \big)
  \tx{,}\quad
  f
  \lmto
  \iota_{\Ga',\chi'} \circ f
  \tx{.}
\end{gather*}

To make this more explicit, we choose bases for each of the inductions~$\Ind_{\Ga'}^{\SL{2}(\ZZ)}\, \chi'$ that are subsets of~$\ga_1,\ldots,\ga_n$.
Since the induction is a permutation representation on this basis up to multiplication by roots of unity, they are, in particular, Hermitian bases with respect to the standard dot product.
We also choose a basis~$\delta_1,\ldots,\delta_s$ of~$\genInd_{\Ga}^{\SL{2}(\ZZ)}(\chi,\Xi)$ that consists of images of~$\ga_1,\ldots,\ga_n$ in the quotient~\eqref{eq:def:genuine_induction}.
For all~$\chi'\in\Xi$ on~$\Ga'$, the homomorphisms~$\iota_{\Ga',\chi'}$ are represented by matrices~$A(\Ga',\chi')$ with respect to these bases.
Taken together, the columns of all the~$A(\Ga',\chi')$ span a subspace~$W$ of~$\Ind_\Ga^{\SL{2}(\ZZ)}\, \chi$.
Let~$v_1,\ldots,v_s \in W^\perp$ be preimages of~$\delta_j$ under the quotient map~$\pi$ that arises from~\eqref{eq:def:genuine_induction}.
They form a basis for~$W^\perp$, since their images are a basis of the quotient space.
We use this basis to construct the desired isomorphism.

\begin{proposition}
\label{prop:vector_valued_modular_forms:genuine_induced_representations}
Let~$A$ be the~$n\times s$ matrix with columns~$v_i$.
We have the isomorphism
\begin{align}
  \label{eq:prop:vector_valued_modular_forms:genuine_induced_representations}
  \genInd
  \defcol
  \rmM_k(\Ga, \chi)
  \mathbin{\Bigg\slash}
  \sum_{\substack{
  \Ga \subsetneq \Ga' \subseteq \SL{2}(\ZZ) \\
  \chi'\in\Xi
    }}
  \mspace{-12mu}
  \rmM_k(\Ga', \chi')
  \lra
  \rmM_k\big( \genInd_\Ga^{\SL{2}(\ZZ)}(\chi,\Xi) \big)
  \tx{,}\quad
  f
  \lmto
  (\lT{\ovsmash{A}} A)^{-1} \lT{\ovsmash{A}} \, \Ind(f)
  \tx{.}
\end{align}
\end{proposition}

\begin{proof}
Left-multiplication by~$A$ gives a map between representations
\begin{gather*}
  \genInd_\Ga^{\SL{2}(\ZZ)}(\chi,\Xi)
  \lhra
  \Ind_\Ga^{\SL{2}(\ZZ)}\, \chi
\end{gather*}
with complement
\begin{gather*}
  \sum_{\substack{
      \Ga \subsetneq \Ga' \subseteq \SL{2}(\ZZ) \\
      \chi' \in \Xi
    }}
  \mspace{-12mu}
  \Ind_{\Ga'}^{\SL{2}(\ZZ)}\,\chi'
  \subseteq
  \Ind_{\Ga}^{\SL{2}(\ZZ)}\,\chi
\end{gather*}
by induction in steps.
The associated quotient map is given by left-multiplication with~$(\lT{\ovsmash{A}} A)^{-1} \lT{\ovsmash{A}}$.
Hence we have the isomorphism
\begin{gather*}
  \rmM_k
  \big(
  \Ind_\Ga^{\SL{2}(\ZZ)}\, \chi
  \big)
  \mathbin{\Bigg\slash}
  \sum_{\substack{
      \Ga \subsetneq \Ga' \subseteq \SL{2}(\ZZ) \\
      \chi' \in \Xi
    }}
  \mspace{-12mu}
  \rmM_k
  \big(
  \Ind_{\Ga'}^{\SL{2}(\ZZ)}\, \chi'
  \big)
  \lra
  \rmM_k\big( \genInd_\Ga^{\SL{2}(\ZZ)}(\chi,\Xi) \big)
\end{gather*}
given by left-multiplication with the same matrix.
Applying \fref{Proposition}{prop:vector_valued_modular_forms:induced_representations} gives the claim.
\end{proof}

\begin{remark}
The quotient in \fref{Proposition}{prop:vector_valued_modular_forms:genuine_induced_representations} is reminiscent of newforms, which can be viewed as the space of modular forms quotiented by the subspace of oldforms.
However, on the left hand side of~\eqref{eq:prop:vector_valued_modular_forms:genuine_induced_representations} we quotient out the subspace of modular forms for larger groups as opposed to the subspace of all oldforms.
For instance, in the case of the trivial representation of~$\Ga_0(p)$ for a prime~$p$, a rescaled level~$1$ form~$f(p \tau)$ contributes to the modular forms of the irreducible quotient, while~$f(\tau)$ does not.
\end{remark}

\subsection{Tensor products of modular forms}%
\label{ssec:vector_valued_modular_forms:tensor_products}

The tensor product is the vector-valued analog of multiplication of scalar-valued modular forms:
\begin{definition}
\label{def:vector_valued_modular_forms:tensor_products}
Let~$f = (f_i)_{1 \leq i \leq m}$ be invariant for~$|_{k,\rho_1}$ and~$g = (g_j)_{1 \leq j \leq n}$ be invariant for~$|_{l,\rho_2}$, for half-integers~$k$ and~$l$.
The~\emph{tensor product of~$f$ and~$g$} is their tensor product as vectors.
In other words, and in a slight abuse of notation,
\begin{gather*}
  f \otimes g
  =
  \bigl( f_i \cdot g_j \bigr)_{\substack{1\leq i\leq m \\ 1\leq j\leq n}}
  \tx{.}
\end{gather*}
\end{definition}

As in the scalar-valued case, the tensor product preserves modularity:

\begin{proposition}
With~$f$ and~$g$ as in \fref{Definition}{def:vector_valued_modular_forms:tensor_products},~$f \otimes g$ is invariant for~$|_{k + l, \rho_1 \otimes \rho_2}$, where~$\rho_1 \otimes \rho_2$ is the tensor product of~$\rho_1$ and~$\rho_2$.
\end{proposition}

\begin{proof}
This follows from the compatibility between the definitions of~$\otimes$ for slash-invariant functions and that for representations, using that the scalar-valued slash action and the action of~$\rho_1$ and~$\rho_2$ intertwine.
\end{proof}

Given that the tensor product is a form of multiplication for vector-valued functions, it is natural to extend \fref{Definition}{def:rankin-cohen}:
\begin{definition}
\label{def:vector_valued_modular_forms:tensor_rankin-cohen}
Let~$f$ and~$g$ be as in \fref{Definition}{def:vector_valued_modular_forms:tensor_products}.
The~\emph{tensor Rankin--Cohen bracket} of~$f$ and~$g$, denoted~$[f,g]_\nu^\otimes$, is given by applying the Rankin--Cohen bracket from \fref{Definition}{def:rankin-cohen} to~$f \otimes g$ elementwise.
Equivalently,
\begin{gather}
  \label{eq:def:vector_valued_modular_forms:tensor_rankin-cohen}
  \big[ f,g \big]_\nu^\otimes
  \defeq
  \sum_{\mu=0}^\nu(-1)^\mu
  \mbinom{k+\nu-1}{\nu-\mu}
  \mbinom{l+\nu-1}{\mu}\,
  f^{(\mu)}
  \otimes
  g^{(\nu-\mu)}
  \tx{.}
\end{gather}
where~$f^{(\mu)}$ is as in \fref{Definition}{def:rankin-cohen}.
\end{definition}

\begin{remark}
The notion of vector-valued Rankin--Cohen brackets has also appeared in~\cite{choie-lee-2016}.
\end{remark}

The modularity of the function in
\eqref{eq:def:vector_valued_modular_forms:tensor_rankin-cohen}
is a consequence of the following.

\begin{lemma}
\label{la:vector_valued_modular_forms:tensor_rankin-cohen_covariance}
Let~$f$ and~$g$ be as in \fref{Definition}{def:vector_valued_modular_forms:tensor_products}.
Then for all~$\nu \ge 0$,
\begin{gather}
  \label{eq:la:vector_valued_modular_forms:tensor_rankin-cohen_covariance}
  \big[ f,g \big]_\nu^\otimes
  \Big|_{k+l+2\nu,\rho_1\otimes\rho_2}\,
  \ga
  =
  \big[ f,g \big]_\nu^\otimes
  \quad\tx{for all\/ }
  \ga\in\Mp{2}(\ZZ)
  \tx{.}
\end{gather}
\end{lemma}

\begin{proof}
Theorem~7.1 from~\cite{cohen-1975} shows that
\begin{gather*}
  \big[
    f_i|_k\ga
    ,\,
    g_j|_l\ga
    \big]_\nu
  =
  \big[
    f_i
    ,\,
    g_j
    \big]_\nu
  \big|_{k+l+2\nu}
  \ga
  \tx{.}
\end{gather*}
Applying this componentwise yields
\begin{gather*}
\big[ f,g \big]_\nu^\otimes
\Big|_{k+l+2\nu}\,
\ga
=
\big[
f|_k\ga
,\,
g|_l\ga
\big]_\nu^\otimes
=
\big[
\rho_1(\ga)f,
\rho_2(\ga)g
\big]_\nu^\otimes
=
\rho_1(\ga)\otimes\rho_2(\ga)
\big[
f,g
\big]_\nu^\otimes
\end{gather*}
using the invariance of~$f$ under~$|_{k,\rho_1}$ and~$g$ under~$|_{l,\rho_2}$.
Hence
\begin{gather*}
\big(\rho_1(\ga)\otimes\rho_2(\ga)\big)^{-1}
\big[ f,g \big]_\nu^\otimes
\Big|_{k+l+2\nu}\,
\ga
=
\big[ f,g \big]_\nu^\otimes
\Big|_{k+l+2\nu,\rho_1\otimes\rho_2}\,
\ga
=
\big[
f,g
\big]_\nu^\otimes
\tx{,}
\end{gather*}
as claimed.
\end{proof}

\begin{remark}
In the proof of the next lemma, we invoke Theorem~5 of~\cite{imamoglu-raum-richter-2014}. In the notation of this paper, in the second paragraph of its proof it is used that the kernel of~$\rho$ is a congruence subgroup to conclude that~$\Gamma$ is a congruence subgroup. This assumption was accidentally omitted in~\cite{imamoglu-raum-richter-2014}.
\end{remark}

\begin{lemma}
\label{la:vector_valued_modular_forms:tensor_rankin-cohen_cusp}
Assume that~$f$ and~$g$ are as in \fref{Definition}{def:vector_valued_modular_forms:tensor_products}, that the kernels of~$\rho_1$ and~$\rho_2$ are congruence subgroups, and that given~$\nu \ge 0$ with~$k + l + 2\nu > 2$, $[f, g]^\otimes_\nu$ satisfies the assumptions in \fref{Definition}{def:holomorphic_projection}.
Then we have
\begin{gather*}
  \pi^\hol_{k+l+2\nu}
  \big(
  \big[
    f,g
    \big]_\nu
  ^\otimes
  \big)
  \in
  \rmM_{k+l+2\nu}
  \big(
  \rho_1
  \otimes
  \rho_2
  \big)
  \tx{.}
\end{gather*}
\end{lemma}

\begin{proof}
Slash-invariance follows from \fref{Lemma}{la:vector_valued_modular_forms:tensor_rankin-cohen_covariance} and the fact that the slash operator intertwines with~$\pi^\hol$ by Theorem~5 of~\cite{imamoglu-raum-richter-2014} (for which we use~$k + l + 2\nu > 2$), using the fact that the kernel of~$\rho_1 \otimes \rho_2$ is a congruence subgroup.
Holomorphicity follows from the definition of~$\pi^\hol$. The growth condition is satisfied by inspection of the Fourier series expansion.
\end{proof}

\subsection{An induced representation for~\texpdf{$\Ga_0(4)$}{Gamma0(4)}}%
\label{ssec:vector_valued_modular_forms:induced_representation_ga04}

In this section, we work out the details for the irreducible induction of the trivial representation of~$\Ga_0(4)$, which will appear in the proof of our main theorem. That is, we have~$\Ga=\Ga_0(4)$ and~$\rho=\bbone$ throughout this section.

We recall that~$\Ga$ has index 6 in~$\SL{2}(\ZZ)$, and we choose the right coset representatives
\begin{gather}
  \label{eq:representation_theory:hurwitz:ind_ga04_coset_representatives}
  \ga_1
  =
  \begin{psmatrix}
    1&0\\
    0&1
  \end{psmatrix}
  \tx{,}\quad
  \ga_2
  =
  \begin{psmatrix}
    1&0\\
    2&1
  \end{psmatrix}
  \tx{,}\quad
  \ga_3
  =
  \begin{psmatrix}
    0&-1\\
    1&0
  \end{psmatrix}
  \tx{,}\quad
  \ga_4
  =
  \begin{psmatrix}
    0&-1\\
    1&2
  \end{psmatrix}
  \tx{,}\quad
  \ga_5
  =
  \begin{psmatrix}
    0&-1\\
    1&1
  \end{psmatrix}
  \tx{,}\quad
  \ga_6
  =
  \begin{psmatrix}
    0&-1\\
    1&3
  \end{psmatrix}
  \tx{.}
\end{gather}
Since~$\rho$ is trivial, the induced representation is a permutation representation of~$\SL{2}(\ZZ)$ given by right-multiplication of~$S^{-1}$ and~$T^{-1}$ on~$\Ga \mathbin{\backslash} \SL{2}(\ZZ)$. Our choice of coset representatives yields the matrix representation
\begin{gather}
  \label{eq:representation_theory:hurwitz:ind_ga04}
  \Ind_{\Ga_0(4)}^{\SL{2}(\ZZ)}\, \bbone
  \defcol\quad
  S
  \lmto
  \begin{psmatrix}
    0 & 0 & 1 & 0 & 0 & 0 \\
    0 & 0 & 0 & 1 & 0 & 0 \\
    1 & 0 & 0 & 0 & 0 & 0 \\
    0 & 1 & 0 & 0 & 0 & 0 \\
    0 & 0 & 0 & 0 & 0 & 1 \\
    0 & 0 & 0 & 0 & 1 & 0
  \end{psmatrix}
  \tx{,}\quad
  T
  \lmto
  \begin{psmatrix}
    1 & 0 & 0 & 0 & 0 & 0 \\
    0 & 1 & 0 & 0 & 0 & 0 \\
    0 & 0 & 0 & 0 & 1 & 0 \\
    0 & 0 & 0 & 0 & 0 & 1 \\
    0 & 0 & 0 & 1 & 0 & 0 \\
    0 & 0 & 1 & 0 & 0 & 0
  \end{psmatrix}
  \tx{.}
\end{gather}
In particular, the elements in~\eqref{eq:representation_theory:hurwitz:ind_ga04_coset_representatives} yield a Hermitian basis for the induction.

We also need~$\Ind_{\Ga_0(2)}^{\SL{2}(\ZZ)}\, \bbone$, which we determine using three of our choices~$\ga_1$,~$\ga_3$,~$\ga_5$ as coset representatives for~$\Ga_0(2) \mathbin{\backslash} \SL{2}(\ZZ)$. This yields the permutation representation
\begin{gather}
  \label{eq:representation_theory:hurwitz:ind_ga02}
  \Ind_{\Ga_0(2)}^{\SL{2}(\ZZ)}\, \bbone
  \defcol\quad
  S
  \lmto
  \begin{psmatrix}
    0 & 1 & 0 \\
    1 & 0 & 0 \\
    0 & 0 & 1
  \end{psmatrix}
  \tx{,}\quad
  T
  \lmto
  \begin{psmatrix}
    1 & 0 & 0 \\
    0 & 0 & 1 \\
    0 & 1 & 0
  \end{psmatrix}
  \tx{.}
\end{gather}

Elements of~$\rmM_k\bigl(  \Ind_{\Ga_0(4)}^{\SL{2}(\ZZ)}\, \bbone \bigr)$ take values in a~$6$-dimensional~$\CC$-vector space, whose basis is naturally indexed by~$\ga_1, \ldots, \ga_6$, and transform under~$|_k$ per the matrices in~\eqref{eq:representation_theory:hurwitz:ind_ga04}.
The space~$\rmS_k(\Ga_0(4))$ is spanned by newforms of level~$\Ga_0(4)$, as well as~$f(\tau)$ and~$f(2\tau)$ for~$f \in \rmS_k(\Ga_0(2))$.
We deduce from this and the list of newforms at LMFDB~\cite{lmfdb} that the spaces of cusp forms for~$\Ga_0(4)$ of weight~$2$ and~$4$ are trivial and those of weight~$6$ and~$8$ are spanned by
\begin{gather*}
  \eta(2\tau)^{12}
  \quad\tx{and}\quad
  \eta(\tau)^8\,
  \eta(2\tau)^8
  \tx{,}\;\;
  \eta(2\tau)^8\,
  \eta(4\tau)^8
  \tx{,}
\end{gather*}
respectively.
By \fref{Proposition}{prop:vector_valued_modular_forms:induced_representations} this yields that
\begin{gather}
  \label{eq:representation_theory:hurwitz:ind_ga04_modular_forms_wt4}
  \dim\, \rmS_4\big( \Ind_{\Ga_0(4)}^{\SL{2}(\ZZ)}\, \bbone \big)
  =
  0
  \tx{,}
\end{gather}
and, after using Fourier series expansions to verify the given equalities, the bases
\begin{gather}
  \label{eq:representation_theory:hurwitz:ind_ga04_modular_forms_wt6}
  \Ind\big( \eta(2\tau)^{12} \big)
  =
  \begin{pmatrix}
    \eta(2\tau)^{12} |_6\, \ga_1 \\
    \eta(2\tau)^{12} |_6\, \ga_2 \\
    \eta(2\tau)^{12} |_6\, \ga_3 \\
    \eta(2\tau)^{12} |_6\, \ga_4 \\
    \eta(2\tau)^{12} |_6\, \ga_5 \\
    \eta(2\tau)^{12} |_6\, \ga_6
  \end{pmatrix}
  =
  \begin{pmatrix}
    \eta(2\tau)^{12}                                    \\
    -\eta(2\tau)^{12}                                   \\
    -\frac{1}{64}\, \eta\big(\frac{\tau}{2}\big)^{12}   \\
    \frac{1}{64}\, \eta\big(\frac{\tau}{2}\big)^{12}    \\
    -\frac{1}{64}\, \eta\big(\frac{\tau+1}{2}\big)^{12} \\
    \frac{1}{64}\, \eta\big(\frac{\tau+1}{2}\big)^{12}
  \end{pmatrix}
  \in
  \rmS_6\big( \Ind_{\Ga_0(4)}^{\SL{2}(\ZZ)}\, \bbone \big)
\end{gather}
and
\begin{gather}
  \label{eq:representation_theory:hurwitz:ind_ga04_modular_forms_wt8}
  \begin{pmatrix}
    \eta(\tau)^8\, \eta(2\tau)^8
    \\
    \eta(\tau)^8\, \eta(2\tau)^8
    \\
    \frac{1}{16}\,
    \eta\big(\frac{\tau}{2}\big)^8\,
    \eta(\tau)^8
    \\
    \frac{1}{16}\,
    \eta\big(\frac{\tau}{2}\big)^8\,
    \eta(\tau)^8
    \\
    \frac{1}{16}\,
    \eta\big(\frac{\tau+1}{2}\big)^8\,
    \eta(\tau+1)^8
    \\
    \frac{1}{16}\,
    \eta\big(\frac{\tau+1}{2}\big)^8\,
    \eta(\tau+1)^8
  \end{pmatrix}
  \tx{,}\quad
  \begin{pmatrix}
    \eta(2\tau)^8\, \eta(4\tau)^8 \\
    \frac{1}{16}\,
    \eta\big(\tau+\frac{1}{2}\big)^8\,
    \eta(2\tau+1)^8
    \\
    \frac{1}{4096}\,
    \eta\big(\frac{\tau}{2}\big)^8\,
    \eta\big(\frac{\tau}{4}\big)^8
    \\
    \frac{1}{4096}\,
    \eta\big(\frac{\tau+2}{4}\big)^8\,
    \eta\big(\frac{\tau}{2}+1\big)^8
    \\
    \frac{1}{4096}\,
    \eta\big(\frac{\tau+1}{4}\big)^8\,
    \eta\big(\frac{\tau+1}{2}\big)^8
    \\
    \frac{1}{4096}\,
    \eta\big(\frac{\tau+3}{4}\big)^8\,
    \eta\big(\frac{\tau+3}{2}\big)^8
  \end{pmatrix}
  \in
  \rmS_8\big( \Ind_{\Ga_0(4)}^{\SL{2}(\ZZ)}\, \bbone \big)
  \tx{.}
\end{gather}

Continuing with representation theory, we next deduce~$\genInd_{\Ga_0(4)}^{\SL{2}(\ZZ)}\, \bbone$. By its definition in~\eqref{eq:def:genuine_induction} we have
\begin{gather*}
  \genInd_{\Ga_0(4)}^{\SL{2}(\ZZ)}\, \bbone
  =
  \Ind_{\Ga_0(4)}^{\SL{2}(\ZZ)}\, \bbone
  \mathbin{\Big\slash}
  \Ind_{\Ga_0(2)}^{\SL{2}(\ZZ)}\, \bbone
  \tx{.}
\end{gather*}
Since~$\ga_i$ and~$\ga_{i+1}$ are in the same coset of~$\Ga_0(2)$ for~$i = 1, 3, 5$, specializing~\eqref{eq:representation_theory:hurwitz:ind_inclusion}, we obtain the inclusion
\begin{gather}
  \label{eq:representation_theory:hurwitz:ind_ga02_to_ga04}
  \Ind_{\Ga_0(2)}^{\SL{2}(\ZZ)}\, \bbone
  \lhra
  \Ind_{\Ga_0(4)}^{\SL{2}(\ZZ)}\, \bbone
  \quad\tx{with matrix }
  \begin{psmatrix}
    1 & 0 & 0 \\
    1 & 0 & 0 \\
    0 & 1 & 0 \\
    0 & 1 & 0 \\
    0 & 0 & 1 \\
    0 & 0 & 1
  \end{psmatrix}
  \tx{.}
\end{gather}
Thus, when we quotient the induction from~$\Ga$ by the one from~$\Ga_0(2)$, we calculate the quotient of the vector space with basis~\eqref{eq:representation_theory:hurwitz:ind_ga04_coset_representatives} modulo the relations~$\ov\ga_1+\ov\ga_2=\ov\ga_3+\ov\ga_4=\ov\ga_5+\ov\ga_6=0$, where we write~$\ov\ga_i$ for the image of~$\ga_i$ in the quotient.
If we choose~$\delta_1 = \ov\ga_1$,~$\delta_2 = \ov\ga_3$, and~$\delta_3 = \ov\ga_5$ to be a basis for the irreducible induction, we obtain the matrices
\begin{gather}
  \label{eq:representation_theory:hurwitz:genind_ga04}
  \genInd_{\Ga_0(4)}^{\SL{2}(\ZZ)}\, \bbone
  \defcol\quad
  S
  \lmto
  \begin{psmatrix}
    0 & 1 &  0 \\
    1 & 0 &  0 \\
    0 & 0 & -1
  \end{psmatrix}
  \tx{,}\quad
  T
  \lmto
  \begin{psmatrix}
    1 &  0 & 0 \\
    0 &  0 & 1 \\
    0 & -1 & 0
  \end{psmatrix}
  \tx{.}
\end{gather}
For example,~$\ga_5$ is sent to~$\ga_6$ by~$S$ in~\eqref{eq:representation_theory:hurwitz:ind_ga04}, which maps to~$-\ov\ga_5$ in the irreducible induction by the above relations. This explains the appearance of~$-1$ in the image of~$S$ in~\eqref{eq:representation_theory:hurwitz:genind_ga04}. From the matrix in~\eqref{eq:representation_theory:hurwitz:ind_ga02_to_ga04} we deduce the $\delta_i$\nbd{}preimages
\begin{gather*}
  v_1
  =
  \mfrac{1}{2}
  \begin{psmatrix} 1 \\ -1 \\ 0 \\ 0 \\ 0 \\ 0 \end{psmatrix}
  \tx{,}\quad
  v_2
  =
  \mfrac{1}{2}
  \begin{psmatrix} 0 \\ 0 \\ 1 \\ -1 \\ 0 \\ 0 \end{psmatrix}
  \tx{,}\quad
  v_3
  =
  \mfrac{1}{2}
  \begin{psmatrix} 0 \\ 0 \\ 0 \\ 0 \\ 1 \\ -1 \end{psmatrix}
  \tx{.}
\end{gather*}

Applying the quotient map to the vector-valued modular form in~\eqref{eq:representation_theory:hurwitz:ind_ga04_modular_forms_wt6} and rescaling by a factor of~$\frac{1}{2}$, we obtain the basis element
\begin{gather}
  \label{eq:representation_theory:hurwitz:genind_ga04_modular_forms_wt6}
  \mfrac{1}{2}
  \begin{pmatrix}
    \eta(2\tau)^{12}
    -
    (-\eta(2\tau)^{12})
    \\
    -\frac{1}{64}\, \eta\big(\frac{\tau}{2}\big)^{12}
    -
    \frac{1}{64}\, \eta\big(\frac{\tau}{2}\big)^{12}
    \\
    -\frac{1}{64}\, \eta\big(\frac{\tau+1}{2}\big)^{12}
    -
    \frac{1}{64}\, \eta\big(\frac{\tau+1}{2}\big)^{12}
  \end{pmatrix}
  =
  \begin{pmatrix}
    \eta(2\tau)^{12}                                  \\
    -\frac{1}{64}\, \eta\big(\frac{\tau}{2}\big)^{12} \\
    -\frac{1}{64}\, \eta\big(\frac{\tau+1}{2}\big)^{12}
  \end{pmatrix}
  \in
  \rmS_6\big( \genInd_{\Ga_0(4)}^{\SL{2}(\ZZ)}\, \bbone \big)
  \tx{.}
\end{gather}
The first modular form in~\eqref{eq:representation_theory:hurwitz:ind_ga04_modular_forms_wt8} maps to zero. The second one maps to the basis element
\begin{gather}
  \label{eq:representation_theory:hurwitz:genind_ga04_modular_forms_wt8}
  \begin{pmatrix}
    \frac{1}{2}\,
    \eta(2\tau)^8\, \eta(4\tau)^8
    -
    \frac{1}{32}\,
    \eta\big(\tau+\frac{1}{2}\big)^8\,
    \eta(2\tau+1)^8
    \\
    \frac{1}{8192}\, \big(
    \eta\big(\frac{\tau}{2}\big)^8\,
    \eta\big(\frac{\tau}{4}\big)^8
    -
    \eta\big(\frac{\tau+2}{4}\big)^8\,
    \eta\big(\frac{\tau}{2}+1\big)^8
    \big)
    \\
    \frac{1}{8192}\, \big(
    \eta\big(\frac{\tau+1}{4}\big)^8\,
    \eta\big(\frac{\tau+1}{2}\big)^8
    -
    \eta\big(\frac{\tau+3}{4}\big)^8\,
    \eta\big(\frac{\tau+3}{2}\big)^8
    \big)
  \end{pmatrix}
  \in
  \rmS_8\big( \genInd_{\Ga_0(4)}^{\SL{2}(\ZZ)}\, \bbone \big)
  \tx{,}
\end{gather}
again after rescaling by~$\frac{1}{2}$.

\subsection{Decomposition of a tensor product}%
\label{ssec:vector_valued_modular_forms:tensor_product_hurwitz}

Recall the representation~$\rhoH$ from \fref{Section}{ssec:preliminaries:class_number_generating_series} associated with vector-valued generating series of Hurwitz class numbers. In this section, we decompose~$\ov{\rhoH} \otimes \rhoH$ into its irreducible subrepresentations, which we have studied in \fref{Section}{ssec:vector_valued_modular_forms:induced_representation_ga04}.

Henceforth, we regard~$\ov{\rhoH} \otimes \rhoH$ as a representation of~$\SL{2}(\ZZ)$; it factors through~$\SL{2}(\ZZ)$ via the quotient map~$\Mp{2}(\ZZ) \ra \SL{2}(\ZZ)$.

\begin{lemma}
\label{la:representation_theory:hurwitz:decomposition}
With~$\rhoH$ as in~\eqref{eq:def:representation:hurwitz}, we have the decomposition into irreducible representations
\begin{gather*}
  \ov\rhoH \otimes \rhoH
  \cong
  \bbone
  \oplus
  \genInd_{\Ga_0(4)}^{\SL{2}(\ZZ)}\, \bbone
  \tx{.}
\end{gather*}
The first and second direct summands are realized on the subspace spanned by the columns of
\begin{gather}
  \label{eq:la:representation_theory:hurwitz:decomposition:bases}
  \begin{psmatrix} 1 \\ 0 \\ 0 \\ 1 \end{psmatrix}
  \quad\tx{and}\quad
  \begin{psmatrix}
    1 & 0 & 0 \\
    0 & 1 & e(\frac{3}{4}) \\
    0 & 1 & e(\frac{1}{4}) \\
    -1 & 0 & 0
  \end{psmatrix}
  \tx{.}
\end{gather}
\end{lemma}

\begin{proof}
Recall that the matrix representation~$\rhoH$ is given in~\eqref{eq:def:representation:hurwitz} by
\begin{gather*}
  S
  \lmto
  \mfrac{1+e(\frac{3}{4})}{2}
  \begin{psmatrix}
    1 & 1 \\
    1 & -1
  \end{psmatrix}
  \tx{,}\quad
  T
  \lmto
  \begin{psmatrix}
    1 &  \\
    & e(\frac{1}{4})
  \end{psmatrix}
  \tx{.}
\end{gather*}
That is, the tensor product~$\ov\rhoH \otimes \rhoH$ is the matrix representation
\begin{gather*}
  S
  \lmto
  \mfrac{1}{2}
  \begin{psmatrix}
    1 &  1 &  1 &  1 \\
    1 & -1 &  1 & -1 \\
    1 &  1 & -1 & -1 \\
    1 & -1 & -1 &  1
  \end{psmatrix}
  \tx{,}\quad
  T
  \lmto
  \begin{psmatrix}
    1 &  \\
    & e(\frac{1}{4}) \\
    &                & e(\frac{3}{4}) &  \\
    &                &                & 1
  \end{psmatrix}
  \tx{,}
\end{gather*}
where we here and later omit some of the zero entries.

The change of basis associated with the bases in~\eqref{eq:la:representation_theory:hurwitz:decomposition:bases} is
\begin{gather*}
  u
  =
  \begin{psmatrix}
    1 &  1 & 0 & 0 \\
    0 &  0 & 1 & e(\frac{3}{4}) \\
    0 &  0 & 1 & e(\frac{1}{4}) \\
    1 & -1 & 0 & 0
  \end{psmatrix}
  \quad\tx{with}\quad
  u^{-1}
  =
  \mfrac{1}{2}
  \begin{psmatrix}
    1 & 0              & 0              &  1 \\
    1 & 0              & 0              & -1 \\
    0 & 1              & 1              &  0 \\
    0 & e(\frac{1}{4}) & e(\frac{3}{4}) &  0
  \end{psmatrix}
  \tx{.}
\end{gather*}
Applying it to the images of~$S$ and~$T$ under~$\ov\rhoH \otimes \rhoH$ we find
\begin{gather*}
  S
  \lmto
  \begin{psmatrix}
    1 \\
    & 0 & 1 &  0 \\
    & 1 & 0 &  0 \\
    & 0 & 0 & -1
  \end{psmatrix}
  \tx{,}\quad
  T
  \lmto
  \begin{psmatrix}
    1 \\
    & 1 &  0 & 0 \\
    & 0 &  0 & 1 \\
    & 0 & -1 & 0
  \end{psmatrix}
  \tx{.}
\end{gather*}
The top left entries correspond to the trivial representation and the bottom right~$3 \times 3$ blocks are precisely the matrices in~\eqref{eq:representation_theory:hurwitz:genind_ga04}.

A calculation with the eigenspaces associated with~$S$ and~$T$ shows that the genuine induction in the statement is irreducible as claimed.
\end{proof}

With this decomposition at hand, we can also determine the properties of spaces of modular forms that will occur in \fref{Section}{sec:proof_main_theorem}.

\begin{corollary}%
\label{cor:representation_theory:hurwitz_modular_forms_support}
Given~$\rhoH$ as in \fref{Section}{ssec:preliminaries:class_number_generating_series}, we have the following statements:
\begin{enumeratearabic}
\item
\label{it:cor:representation_theory:hurwitz_modular_forms_support:weight4}
We have~$\dim\, \rmS_4(\ov\rhoH \otimes \rhoH) = 0$.

\item
\label{it:cor:representation_theory:hurwitz_modular_forms_support:weight6}
Further, for~$f = (f_i)_{1 \le i \le 4} \in \rmS_6(\ov\rhoH \otimes \rhoH)$ we have~$c(f_1; n) = c(f_4; n) = 0$ if\/~$n \in 2 \ZZ$.

\item
\label{it:cor:representation_theory:hurwitz_modular_forms_support:weight6to10}
Finally, for~$k \in \{4, 6, 8, 10, 14\}$ and~$f = (f_i)_{1 \le i \le 4} \in \rmS_k(\ov\rhoH \otimes \rhoH)$, we have~$c(f_1; n) = -c(f_4; n)$ for all~$n \in \ZZ$.
\end{enumeratearabic}
\end{corollary}

\begin{proof}
We combine \fref{Lemma}{la:representation_theory:hurwitz:decomposition} with the calculations in \fref{Section}{ssec:vector_valued_modular_forms:induced_representation_ga04}. Specifically, the first statement follows from~$\dim\, \rmS_4( \genInd_{\Ga_0(4)}^{\SL{2}(\ZZ)}\, \bbone) ) = 0$, which is a consequence of~\eqref{eq:representation_theory:hurwitz:ind_ga04_modular_forms_wt4}.
The second statement follows from~\eqref{eq:representation_theory:hurwitz:genind_ga04_modular_forms_wt6}, which implies that if~$k = 6$ both~$f_1$ and~$f_4$ are multiples of~$\eta(2 \tau)^{12}$.
The third statement uses the fact that~$\dim\, \rmS_k = \dim\, \rmS_k(\bbone) = 0$ for the given~$k$.
\end{proof}

\section{Fourier coefficients of holomorphic projections}%
\label{sec:holomorphic_projection_fourier_coefficients}
\subsection{Hypergeometric functions}%
\label{ssec:holomorphic_projection_fourier_coefficients:special_functions}

In this section, all square roots are taken to be the principal root.

We evaluate some hypergeometric functions, which will arise in the holomorphic projection in \fref{Section}{ssec:holomorphic_projection_fourier_coefficients:calculations_constant_term}. We define them for complex~$z$ with~$|z| < 1$ via the Gauss series
\begin{gather*}
  {}_2F_1\big( a, b, c; z \big)
  \defeq
  \sum_{n = 0}^\infty
  \frac{(a)_n\, (b)_n}{(c)_n\, n!}\, z^n
  \tx{,}
\end{gather*}
where~$(a)_n = a (a+1) \cdots (a+n-1)$ is the rising factorial.

\begin{lemma}
\label{la:holomorphic_projection:hurwitz:hypergeometric_intermediate}
  Given an integer~$m \ge 0$ and $k\in\frac{1}{2}+\ZZ$, we have for~$z\in\CC^*\setminus[1,\infty)$ that
\begin{gather*}
  {}_2F_1\big( 1, k + m, 2; z \big)
  =
  \mfrac{1}{1-k-m}\, z^{-1} \big( 1 - (1-z)^{1-k-m} \big)
  \tx{.}
\end{gather*}
\end{lemma}

\begin{proof}
  By 15.6.1 in~\cite{nist-dlmf-1-2-4}, we have
  \begin{gather*}
    {}_2F_1(1,k+m;2;z)
    =
    {}_2F_1(k+m,1;2;z)
    =
    \int_0^1
    (1-zt)^{-(k+m)}
    \,\rmd\!t
    \tx{,}
  \end{gather*}
  and the claimed formula follows after evaluating the integral.
\end{proof}

Extending \fref{Lemma}{la:holomorphic_projection:hurwitz:hypergeometric_intermediate} to more general third arguments, we have the following:

\begin{lemma}
\label{la:holomorphic_projection:hurwitz:hypergeometric}
  For integers~$m, n \ge 0$ and~$k\in\frac{1}{2}+\ZZ$, we have for~$z\in\CC^*\setminus[1,\infty)$ that
\begin{gather*}
  {}_2F_1\big( 1, k + m, 2 + n; z \big)
  =
  (-1)^{n+1}
  \mbinom{1-k-m+n}{n+1}^{-1}
  \Bigl(
  \mfrac{(1 - z)^{1-k-m+n}}{z^{n+1}}
  -
  \sum_{j=0}^n
  \mbinom{k+m-1}{n-j}
  \mfrac{(1-z)^j}{z^{j+1}}
  \Bigr)
  \tx{.}
\end{gather*}
\end{lemma}

Note that the first binomial coefficient on the right is nonzero since~$k$ is half-integral.

\begin{proof}
Formula~15.5.6 together with~15.5.10 in~\cite{nist-dlmf-1-2-4} implies that
\begin{gather*}
  {_2F_1}\big( 1,k + m, 2 + n ;z \big)
  =
  (n+1)
  \mfrac{\Ga(2-k-m)}{\Ga(2-k-m+n)}\,
  (1-z)^{1-k-m+n}\,
  \mfrac{\rmd^{n}}{\rmd\mspace{-2mu} z^{n}}
  \Big(
  (1-z)^{k+m-1}\,
  {_2F_1}\big( 1,k+m,2;z \big)
  \Big)
  \tx{.}
\end{gather*}
We insert the expression from \fref{Lemma}{la:holomorphic_projection:hurwitz:hypergeometric_intermediate} into the right hand side and apply Leibniz's differentiation rule to obtain
\begin{align*}
      &
  \mfrac{n + 1}{1 - k - m}
  \mfrac{\Ga( 2 - k - m )}{\Ga( 2 - k - m + n )}
  (1 - z) ^ {1 - k - m + n}
  \\
      & \mspace{120mu}\cdot
  \Big(
  \mfrac{(-1)^{n+1} \Ga(n+1)}{z^{n+1}}
  +
  \sum_{j=0} ^ n
  \mbinom{n}{j}
  \mfrac{(-1)^j \Ga(j+1)}{z^{j+1}}
  \mfrac{\Ga(k + m )}{\Ga(k + m - n +j)}
    (-1)^{n-j}
    (1 - z) ^ { (k + m - 1 - n + j) }
  \Big)
  \\
  ={} &
  \mfrac{(-1)^n}{1-k-m}
  \mfrac{\Ga(n+2) \Ga(2-k-m)}{\Ga(2-k-m+n)}
  \Big(
  -\mfrac{(1-z)^{1-k-m+n}}{z^{n+1}}
  +
  \sum_{j=0}^n
  \mfrac{\Ga(k+m)}{\Ga(k+m-n+j) \Ga(n-j+1)}
  \cdot
  \mfrac{(1-z)^j}{z^{j+1}}
  \Big)
  \tx{,}
\end{align*}
after rewriting~$\binom{n}{j}$ in terms of~$\Ga$ and canceling common factors.
\end{proof}

\subsection{Calculations of constant terms}%
\label{ssec:holomorphic_projection_fourier_coefficients:calculations_constant_term}

In this subsection, we assume that~$k+l+2\nu$ is an integer greater than or equal to~$2$.

Here, ``constant terms'' refers to terms in~$\big[ E^-_{\frac{3}{2}} , \theta_{\frac{1}{2}} \big]^\otimes_\nu$ arising from the term~$y^{-\frac{1}{2}}$ in~$E^-_{\frac{3}{2}}$.
These terms have the general form~$\big[y^{-\kappa},e(\td{n}\tau)\big]_\nu$ for~$\td{n} \geq 0$, so it is enough to calculate~$\pi^\hol_{2+2\nu}$ for these terms. For~$\td{n} > 0$, we use the integral in~\eqref{eq:def:holomorphic_projection}:

\begin{lemma}
\label{la:holomorphic_projection:fourier_term_constant}
For half-integers~$k, l$, an integer~$\nu \ge 0$, a rational number~$\td{n} > 0$, and a real number~$\kappa < k+l-1+2\nu$:
\begin{gather}
  \label{eq:la:holomorphic_projection:fourier_term_constant}
  \pi^\hol_{k+l+2\nu}\bigl(
  y^{-\kappa}\,
  e( \td{n} \tau )
  \bigr)
  =
  (4\pi)^{\kappa}\,
  \mfrac{\Ga(k+l+2\nu-1-\kappa)}{\Ga(k+l+2\nu-1)}\,
  \td{n}^\kappa\,
  e(\td{n} \tau)
  \tx{.}
\end{gather}
\end{lemma}

\begin{proof}
\fref{Assumption}{it:def:holomorphic_projection:constant} in \fref{Definition}{def:holomorphic_projection} is vacuous, and \fref{Assumption}{it:def:holomorphic_projection:growth} is satisfied by the upper bound on~$\kappa$.
From~\eqref{eq:def:holomorphic_projection}, the left hand side of~\eqref{eq:la:holomorphic_projection:fourier_term_constant} in the lemma equals~$c(\td{n})\, e(\td{n} \tau)$, where
\begin{gather*}
  c(\td{n})
  =
  \mfrac{( 4\pi \td{n} )^{k+l+2\nu-1}}{\Ga(k+l+2\nu-1)}\,
  \int\limits_0^\infty e^{-4\pi \td{n} y} y^{k+l+2\nu-\kappa-2} \intrmd y
  =
  \mfrac{( 4\pi \td{n} )^{k+l+2\nu-1}}{\Ga(k+l+2\nu-1)}
  \cdot
  \mfrac{\Ga(k+l+2\nu-\kappa-1)}{( 4\pi \td{n} )^{k+l+2\nu-\kappa-1}}
\end{gather*}
after substituting, using the Haar measure~$\intrmd y \slash y$. After further simplification this yields the right hand side of~\eqref{eq:la:holomorphic_projection:fourier_term_constant}.
\end{proof}

Next, we generalize from the product to the Rankin--Cohen bracket:

\begin{proposition}
\label{prop:holomorphic_projection:fourier_term_rankin_cohen_constant}
For~$k$,~$l$,~$\nu$, and~$\td{n}$ as in \fref{Lemma}{la:holomorphic_projection:fourier_term_constant} and~$\kappa<k+l-1+\nu$,
\begin{alignat*}{2}
   &
  \pi^\hol_{k+l+2\nu}\big( \big[
      y^{-\kappa}
      ,\,
      e(\td{n}\tau)
      \big]_\nu
  \big)
   &   & =
  C\,
  \td{n}^{\kappa+\nu}\,
  e(\td{n}\tau)
  \tx{,}
\end{alignat*}
with
\begin{gather}
  \label{eq:prop:holomorphic_projection:fourier_term_rankin_cohen_constant}
  C
  =
  (4\pi)^\kappa\,
  \mfrac{\Ga(k+l+2\nu-1-\kappa)}{\Ga(k+l+2\nu-1)}\,
  \mbinom{k+\nu-1-\kappa}{\nu}
  \mbinom{k+l+2\nu-2}{\nu}
  \mbinom{k+l+2\nu-2-\kappa}{\nu}^{-1}
  \tx{.}
\end{gather}
\end{proposition}

\begin{proof}
As usual~$\frac{\partial}{\partial\tau} = \frac{1}{2} (\frac{\partial}{\partial x} - i \frac{\partial}{\partial y})$. We find that
\begin{gather*}
  \big[
    y^{-\kappa}
    ,\,
    e(\td{n} \tau)
    \big]_\nu
  =
  \sum_{\mu=0}^\nu
  (-1)^\mu
  \mbinom{k+\nu-1}{\nu-\mu}
  \mbinom{l+\nu-1}{\mu}\,
  \mfrac{1}{(-4\pi)^\mu}
  \mfrac{\Ga(1-\kappa)}{\Ga(1-\kappa-\mu)}\,
  \td{n}^{\nu-\mu}\,
  y^{-\kappa-\mu}\,
  e(\td{n}\tau)
  \tx{.}
\end{gather*}
\fref{Lemma}{la:holomorphic_projection:fourier_term_constant} can then be applied to every term in the sum, giving
\begin{align*}
      &
  \pi^\hol_{k+l+2\nu}\big(\big[ y^{-\kappa},\, e(\td{n}\tau) \big]_\nu\big)
  \\
  ={} &
  \sum_{\mu=0}^\nu
  (-1)^\mu
  \mbinom{k+\nu-1}{\nu-\mu}
  \mbinom{l+\nu-1}{\mu}
  \mfrac{1}{(-4\pi)^\mu}
  \mfrac{\Ga(1-\kappa)}{\Ga(1-\kappa-\mu)}\;
  (4\pi)^{\kappa+\mu}
  \mfrac{\Ga(k+l+2\nu-1-\kappa-\mu)}{\Ga(k+l+2\nu-1)}\,
  \td{n}^{\nu-\mu}\,
  \td{n}^{\kappa+\mu}\,
  e(\td{n}\tau)
  \\
  ={} &
  (4\pi)^{\kappa}
  \mfrac{\Ga{(k+l+2\nu-1-\kappa)}}{\Ga{(k+l+2\nu-1)}}
  \sum_{\mu=0}^\nu
  \mbinom{k+\nu-1}{\nu-\mu}
  \mbinom{l+\nu-1}{\mu}
  \mbinom{-\kappa}{\mu}
  \mbinom{k+l+2\nu-2-\kappa}{\mu}^{-1}\,
  \td{n}^{\kappa+\nu}\,
  e(\td{n}\tau)
  \tx{.}
\end{align*}
Now~\eqref{eq:prop:holomorphic_projection:fourier_term_rankin_cohen_constant} follows from Identity~(11.3) from~\cite{gould-1972}, reproduced here for convenience:
\begin{gather*}
  \sum_{k=0}^n
  \mbinom{x}{k}
  \mbinom{y}{k}
  \mbinom{z}{n-k}
  \mbinom{x+y+z}{k}^{-1}
  =
  \mbinom{x+z}{n}
  \mbinom{y+z}{n}
  \mbinom{x+y+z}{n}^{-1}
  \tx{,}
\end{gather*}
where we allow ourselves to use~$k$,~$n$,~$x$,~$y$,~$z$ as in the reference.
\end{proof}

We now deal with the case~$\td{n} = 0$:

\begin{proposition}
\label{prop:holomorphic_projection:fourier_term_rankin_cohen_zeroth}
For~$\nu \ge 0$ and~$\kappa > 0$,
\begin{gather*}
  \pi^\hol_{k+l+2\nu}\big( \big[
    y^{-\kappa},\,
    1
    \big]_\nu
  \big)
  =
  0
  \tx{.}
\end{gather*}
\end{proposition}

\begin{proof}
From~\eqref{eq:def:rankin-cohen} we deduce that
\begin{gather*}
  \big[
    y^{-\kappa},\,
    1
    \big]_\nu
  =
  \mfrac{1}{(4\pi)^\nu} \,
  \mbinom{l+\nu-1}{\nu}
  \mfrac{\Ga(1-\kappa)}{\Ga(1-\kappa-\nu)}
  y^{-\kappa-\nu}
  \tx{,}
\end{gather*}
which vanishes as~$y\ra\infty$, verifying \fref{Assumption}{it:def:holomorphic_projection:constant} in \fref{Definition}{def:holomorphic_projection}; \fref{Assumption}{it:def:holomorphic_projection:growth} is vacuous. 
\end{proof}
\subsection{Calculations of non-constant terms}%
\label{ssec:holomorphic_projection_fourier_coefficients:calculations_nonconstant_terms}

As in \fref{Lemma}{la:holomorphic_projection:fourier_term_constant}, we start with evaluating the integral in~\eqref{eq:def:holomorphic_projection}.

\begin{lemma}
\label{la:holomorphic_projection:fourier_term_nonconstant}
For half-integers~$k, l$ with~$l > 0$ and~$k + l$ an integer greater than or equal to~$2$, integers~$\mu, \nu \ge 0$ with~$\mu \le \nu$, and rational numbers~$n < 0$,~$\td{n} > 0$ such that~$\td{n} > |n|$,
\begin{gather}
  \label{eq:holomorphic_projection:fourier_term_nonconstant}
  \begin{aligned}
      &
  \pi^\hol_{k+l+2\nu}\bigl(
  \Ga(1-k-\mu, 4 \pi |n| y)\,
  e\big( (n+\td{n}) \tau \big)
  \bigr)
  \\[.2\baselineskip]
  ={} &
  \mfrac{\Ga{(l+2\nu-\mu)}}{\Ga(k+l+2\nu)}\;
  \big( 1+ \tfrac{n}{\td{n}} \big)^{l+2\nu+k-1}
  \big| \tfrac{n}{\td{n}} \big|^{1-k-\mu}\,
  {_2F_1} \big(1,l+2\nu-\mu,k+l+2\nu;1+\tfrac{n}{\td{n}}\big)\;
  e\big( (n+\td{n}) \tau \big)
  \tx{,}
  \end{aligned}
\end{gather}
\end{lemma}

\begin{proof}
The defining expression for~$c(m)$ in~\eqref{eq:def:holomorphic_projection} yields
\begin{align*}
  c(n+\td{n})
   & =
  \mfrac{\big( 4 \pi (n+\td{n})\big)^{k+l+2\nu-1}}{\Ga(k+l+2\nu-1)}\,
  \int\limits_0^\infty \Ga(1-k-\mu,4\pi|n|y) e^{-4\pi(n+\td{n})y} y^{k+l+2\nu-2} \,\intrmd y
  \\
   & =
  \mfrac{\big( 4 \pi (n+\td{n})\big)^{k+l+2\nu-1}}{\Ga(k+l+2\nu-1)}
  \cdot
  \mfrac{(4\pi|n|)^{1-k-\mu}\, \Ga(l+2\nu-\mu)}
  {(k+l+2\nu-1)\, (4\pi \td{n})^{l+2\nu-\mu}}\,
  {_2F_1}\big(1,l+2\nu-\mu,k+l+2\nu;1+\tfrac{n}{\td{n}}\big)
\end{align*}
after rearranging, where the second line follows from~(6.455.1) of~\cite{gradshteyn-ryzhik-2007} with
\begin{gather*}
  \alpha \leftarrow 4 \pi |n|
  \tx{,}\;
  \beta \leftarrow 4 \pi (n + \td{n})
  \tx{,}\quad
  \mu \leftarrow k + l + 2\nu - 1
  \tx{,}\;
  \nu \leftarrow 1 - k - \mu
  \tx{.}
\end{gather*}
The assumptions of~(6.455.1) hold, since~$l > 0$,~$k + l > 1$, and~$\mu \le \nu$. The lemma follows after further simplification.
\end{proof}

\begin{proposition}
\label{prop:holomorphic_projection:fourier_term_rankin_cohen_nonconstant}
  In addition to the assumptions of \fref{Lemma}{la:holomorphic_projection:fourier_term_nonconstant}, assume that~$k\in\frac{1}{2}+\ZZ$ and $k + l = 2$.
Then
\begin{multline*}
  \pi^\hol_{k+l+2\nu}\big( \big[
      \Ga(1-k, 4 \pi |n| y)
      e(n \tau),\,
      e(\td{n} \tau)
      \big]_\nu
  \big)
  =
  \sum_{\substack{t \in \frac{1}{2} + \ZZ \tx{ or} \\ t \in \{0,\ldots,\nu\}}}
  C_t\,
  |n|^t \td{n}^{\nu-t}\,
  e\big( (n+\td{n})\tau \big)
  \qquad\tx{with}
  \\
  C_t
  =
  (-1)^{2t}\,
  \Ga(1-k)\,
  \mbinom{k+\nu-1}{\nu-t}
  \mbinom{\nu-k+1}{t}
  \tx{.}
\end{multline*}
\end{proposition}

\begin{proof}
If
$
  p(\tau)
  \defeq
  \Ga(1-k,4\pi|n|y)\, e(n\tau)
  \tx{,}
$
one can show that (see Lemma 4.5 of~\cite{mertens-2016}, for example)
\begin{gather*}
  p^{(\mu)}(\tau)
  =
  (-1)^\mu \mfrac{\Ga(1-k)}{\Ga(1-k-\mu)}\,
  |n|^\mu\,
  \Ga(1-k-\mu,4\pi|n|y)\, e(n\tau)
  \tx{,}
\end{gather*}
where~$p^{(\mu)}$ is the derivative of~$p$ as in \fref{Definition}{def:rankin-cohen}.
Thus, by the definition of Rankin--Cohen brackets in~\eqref{eq:def:rankin-cohen} and \fref{Lemma}{la:holomorphic_projection:fourier_term_nonconstant}, we have to evaluate
\begin{align}
  \nonumber
      &
  \pi^\hol_{k+l+2\nu}\big( \big[
      p(\tau),\,
      e(\td{n} \tau)
      \big]_\nu
  \big)
  \\
  \nonumber
  ={} &
  \sum_{\mu=0}^\nu (-1)^\mu
  \mbinom{k+\nu-1}{\nu-\mu}
  \mbinom{l+\nu-1}{\mu}\;
  \pi^\hol_{k+l+2\nu}\big(
  p^{(\mu)}(\tau) \cdot \td{n}^{\nu-\mu} e(\td{n}\tau)
  \big)
  \\
  \nonumber
  ={} &
  \sum_{\mu=0}^\nu
  \mfrac{\Ga(1-k)}{\Ga(1-k-\mu)}
  \mbinom{k+\nu-1}{\nu-\mu}
  \mbinom{l+\nu-1}{\mu}\,
  |n|^\mu\td{n}^{\nu-\mu}\;
  \pi^\hol_{k+l+2\nu}\big(\Ga(1-k-\mu,4\pi|n|y)\, e\bigl((n+\td{n})\tau\bigr)\big)
  \\
  \label{eq:prf:prop:holomorphic_projection:fourier_term_rankin_cohen_nonconstant:first_expansion}
  ={} &
  \sum_{\mu=0}^\nu
  \mfrac{\Ga(1-k)}{\Ga(1-k-\mu)}\,
  \mbinom{k+\nu-1}{\nu-\mu}
  \mbinom{l+\nu-1}{\mu}\,
  |n|^\mu \td{n}^{\nu-\mu}\;
  C(\mu; n,\td{n})\;
  e\big((n+\td{n})\tau\big)
  \tx{,}
\end{align}
where
\begin{gather*}
  C(\mu,n,\td{n})
  =
  \mfrac{\Ga{(l+2\nu-\mu)}}{\Ga(k+l+2\nu)}\;
  \big( 1+ \tfrac{n}{\td{n}} \big)^{l+2\nu+k-1}
  \big| \tfrac{n}{\td{n}} \big|^{1-k-\mu}\,
  {_2F_1} \big(1,l+2\nu-\mu,k+l+2\nu;1+\tfrac{n}{\td{n}}\big)
\end{gather*}
is the constant appearing in \fref{Lemma}{la:holomorphic_projection:fourier_term_nonconstant}. Since~$n < 0$, we can replace~$n$ by~$-|n|$ and work with positive~$|n|$ and~$\td{n}$ in the proof.

First, we deal with the hypergeometric terms appearing in~$C(\mu,n,\td{n})$.
Using that~$k+l=2$, we apply~15.8.1 from~\cite{nist-dlmf-1-2-4} to write
\begin{gather*}
  {_2F_1}\big(1,l+2\nu-\mu,2+2\nu;\, 1-\tfrac{|n|}{\td{n}}\big)
  =
  \tfrac{\td{n}}{|n|}
  \cdot
  {_2F_1}\big(1,k+\mu,2+2\nu;\, 1-\tfrac{\td{n}}{|n|}\big)
  \tx{.}
\end{gather*}
The right-hand side can be described by \fref{Lemma}{la:holomorphic_projection:hurwitz:hypergeometric} with~$m \leftarrow \mu$,~$n \leftarrow 2 \nu$, and~$z \leftarrow 1 - \frac{\td{n}}{|n|}$.
Since~$\td{n} > |n|$, we have~$z < 0$.
The leading sign in \fref{Lemma}{la:holomorphic_projection:hurwitz:hypergeometric} is negative, and we absorb it into the main factor there. Further, we combine the inverse binomial coefficient that appears in \fref{Lemma}{la:holomorphic_projection:hurwitz:hypergeometric} with the binomial coefficients in~\eqref{eq:prf:prop:holomorphic_projection:fourier_term_rankin_cohen_nonconstant:first_expansion}.
Hence, to prove the proposition, we have to match coefficients in
\begin{align*}
      &
  \sum_{t\in\frac{1}{2} \ZZ}
  C_t\,
  |n|^t \td{n}^{\nu-t}\,
  \\
  ={} &
  \Ga(1-k)
  \sum_{\mu=0}^\nu
  \mbinom{k+\nu-1}{\nu-\mu}
  \mbinom{1-k+\nu}{\mu}
  \big(1-\tfrac{|n|}{\td{n}}\big)^{2\nu+1}
  |n|^{-k}
  \td{n}^{\nu+k}\;
  \Bigg(
  -\mfrac{\big( \frac{\td{n}}{|n|} \big) ^ { 1-k-\mu+2\nu }}
  {\big( 1 - \frac{\td{n}}{|n|}\big)^{2\nu+1}}
  +
  \sum_{j=0}^{2\nu}
  \mbinom{k+\mu-1}{2\nu-j}
  \mfrac{\big( \frac{\td{n}}{|n|} \big)^{j}}
  {\big( 1 - \frac{\td{n}}{|n|}\big)^{j+1}}
  \Bigg)
  \\
  ={} &
  \Ga(1-k)
  \td{n}^\nu
  \sum_{\mu=0}^\nu
  \mbinom{k+\nu-1}{\nu-\mu}
  \mbinom{1-k+\nu}{\mu}
  \Big(
  \big(\tfrac{|n|}{\td{n}}\big)^\mu
  -
  \sum_{j=0}^{2\nu}
  \mbinom{k+\mu-1}{j}
  \big(\tfrac{|n|}{\td{n}}\big)^{1-k}
  \big(\tfrac{|n|}{\td{n}}-1\big)^j
  \Big)
  \tx{,}
\end{align*}
where we have canceled out powers of~$|n|$ and~$\td{n}$, grouped together powers of~$1-\frac{\td{n}}{|n|}$ and~$1-\frac{|n|}{\td{n}}$, and replaced~$j$ by~$2\nu-j$.

To handle the sum over~$\mu$, we set~$y=\frac{|n|}{\td{n}}$. The problem now reduces to finding the coefficients~$C'_t = C_t \slash \Ga(1-k)$ of
\begin{gather}
  \label{eq:prop:holomorphic_projection:fourier_term_rankin_cohen_nonconstant:polynomial}
  \sum_{t \in \frac{1}{2}\ZZ} C'_t\, y^t
  =
  \sum_{\mu=0}^\nu
  \mbinom{k+\nu-1}{\nu-\mu}
  \mbinom{1-k+\nu}{\mu}
  \Bigg(
  y^\mu
  -
  \sum_{j=0}^{2\nu}
  \mbinom{k+\mu-1}{j}
  y^{1-k}
  (y-1)^j
  \Bigg)
  \tx{.}
\end{gather}
Observe that~$\mu$ is an integer, while the sum over~$j$ contains only half-integer powers of~$y$. Hence the integer powers of~$y$ have coefficients
\begin{gather}
  \label{eq:prop:holomorphic_projection:fourier_term_rankin_cohen_nonconstant:c_half_integral}
  C'_t
  =
  \mbinom{k+\nu-1}{\nu-t}
  \mbinom{1-k+\nu}{t}
  \tx{,}
  \quad
  \tx{for } 0\le t\le \nu,
  t \in \ZZ
  \tx{.}
\end{gather}
Note that if~$t \in \ZZ$ and~$t < 0$ or~$t > \nu$, we have~$C'_t = 0$.

As for half-integer powers of~$y$, this requires evaluating a double sum; expanding the sum over~$j$ and collecting powers of~$y$ yields
\begin{align*}
  C'_t
   & =
  -\sum_{\mu=0}^\nu
  \mbinom{k+\nu-1}{\nu-\mu}
  \mbinom{1-k+\nu}{\mu}\,
  \sum_{j=t+k-1}^{2\nu}
  (-1)^{t-j+k-1}
  \mbinom{k+\mu-1}{j}
  \mbinom{j}{t+k-1}
  \\
   & =
  -\sum_{j=t+k-1}^{2\nu}
  (-1)^{t-j+k-1}
  \mbinom{j}{t+k-1}\,
  \sum_{\mu=0}^\nu
  \mbinom{k+\nu-1}{\nu-\mu}
  \mbinom{1-k+\nu}{\mu}
  \mbinom{k+\mu-1}{j}
  \tx{,}
  \qquad
  \tx{for } t \in \tfrac{1}{2} + \ZZ
  \tx{,}
\end{align*}
after switching sums. This crucially lets us apply the Vandermonde convolution and the identity
\begin{gather}
  \label{eq:cor:holomorphic_projection:fourier_term_rankin_cohen_nonconstant:binomial_product}
  \mbinom{\alpha}{\beta}
  \mbinom{\beta}{\gamma}
  =
  \mbinom{\alpha}{\gamma}
  \mbinom{\alpha-\gamma}{\beta-\gamma}
\end{gather}
to simplify the sum over~$\mu$:
\begin{align*}
                                                                                                               &
  \sum_{\mu=0}^\nu
  \mbinom{k+\nu-1}{\nu-\mu}
  \mbinom{1-k+\nu}{\mu}
  \mbinom{k+\mu-1}{j}
  =
  \sum_{\mu=0}^\nu
  \mbinom{k+\nu-1}{k+\mu-1}
  \mbinom{1-k+\nu}{\mu}
  \mbinom{k+\mu-1}{j}
  \\
  \underset{\eqref{eq:cor:holomorphic_projection:fourier_term_rankin_cohen_nonconstant:binomial_product}}{=}{} &
  \mbinom{k+\nu-1}{j}
  \sum_{\mu=0}^\nu
  \mbinom{k+\nu-1-j}{k+\mu-1-j}
  \mbinom{1-k+\nu}{\mu}
  =
  \mbinom{k+\nu-1}{j}
  \sum_{\mu=0}^\nu
  \mbinom{k+\nu-1-j}{\nu-\mu}
  \mbinom{1-k+\nu}{\mu}
  =
  \mbinom{k+\nu-1}{j}
  \mbinom{2\nu-j}{\nu}
  \tx{,}
\end{align*}
where we apply the Vandermonde identity in the last equality.
Summing over~$j$ then gives
\begin{align*}
  C'_t
   &
  \underset{\hphantom{\eqref{eq:cor:holomorphic_projection:fourier_term_rankin_cohen_nonconstant:binomial_product}}}{=}
  -\sum_{j=t+k-1}^{2\nu}
  (-1)^{t+k-1-j}
  \mbinom{j}{t+k-1}
  \mbinom{k+\nu-1}{j}
  \mbinom{2\nu-j}{\nu}
  \\
   &
  \underset{\eqref{eq:cor:holomorphic_projection:fourier_term_rankin_cohen_nonconstant:binomial_product}}{=}
  -\mbinom{k+\nu-1}{t+k-1}
  \sum_{j=t+k-1}^{2\nu}
  (-1)^{t+k-1-j}
  \mbinom{k+\nu-1-t-k+1}{j-t-k+1}
  \mbinom{2\nu-j}{\nu}
  \\
   &
  \underset{\hphantom{\eqref{eq:cor:holomorphic_projection:fourier_term_rankin_cohen_nonconstant:binomial_product}}}{=}
  -\mbinom{k+\nu-1}{t+k-1}
  \sum_{j=0}^{2\nu-t-k+1}
  (-1)^j
  \mbinom{2\nu-t-k+1-j}{\nu}
  \mbinom{\nu-t}{j}
  =
  -\mbinom{k+\nu-1}{t+k-1}
  \mbinom{\nu-k+1}{\nu-t-k+1}
  \tx{,}
\end{align*}
where the last equality follows from identity~(5a) in~\cite{riordan-1968}, given below:
\begin{gather*}
  \sum_{k=0}^n
  (-1)^k
  \mbinom{n-k}{m}
  \mbinom{p}{k}
  =
  \mbinom{n-p}{n-m}
\end{gather*}
with~$k$,~$m$,~$n$, and~$p$ the variables from reference~\cite{riordan-1968}.

Thus, we have determined~$C'_t$ for all~$t \in \frac{1}{2} \ZZ$, finishing the proof.
\end{proof}
\subsection{Calculations specific to Hurwitz class numbers}%
\label{ssec:holomorphic_projection_fourier_coefficients:hurwitz_class_numbers}

We now set~$k = \frac{3}{2}$,~$l = \frac{1}{2}$, and~$\kappa = \frac{1}{2}$, which further simplifies the expressions in \fref{Proposition}{prop:holomorphic_projection:fourier_term_rankin_cohen_constant} and \fref{Proposition}{prop:holomorphic_projection:fourier_term_rankin_cohen_nonconstant}.

\begin{lemma}
\label{la:holomorphic_projection_fourier_coefficients:hurwitz_class_numbers:simplification}
For an integer~$\nu \ge 0$, and rational numbers~$n < 0$ and~$\td{n} > 0$ such that~$\td{n} > |n|$,
\begin{align*}
  \pi^\hol_{2+2\nu}\big( \big[
      y^{-\frac{1}{2}},\,
      e(\td{n} \tau)
      \big]_\nu
  \big)
   & =
  2\pi
  \mbinom{\nu-\frac{1}{2}}{\nu}\,
  \big( \sqrt{\td{n}} \big)^{2\nu+1}
  \,
  e(\td{n} \tau)
  \tx{,}
  \\
  \pi^\hol_{2+2\nu}\big( \big[
      \Ga\big(-\tfrac{1}{2}, 4 \pi |n| y\big)
      e(n \tau),\,
      e(\td{n} \tau)
      \big]_\nu
  \big)
   & =
  2\sqrt{\pi}
  \mbinom{\nu-\frac{1}{2}}{\nu}\,
  \frac{1}{\sqrt{|n|}}
  \big(
  \sqrt{\td{n}}
  -
  \sqrt{|n|}
  \big)
  ^{2\nu+1}
  \,
  e\big( (n+\td{n})\tau \big)
  \tx{.}
\end{align*}
\end{lemma}

\begin{proof}
For the first claim, \fref{Proposition}{prop:holomorphic_projection:fourier_term_rankin_cohen_constant} reduces to
\begin{gather*}
  \pi^\hol_{2+2\nu}\big( \big[
      y^{-\frac{1}{2}}
      ,\,
      e(\td{n}\tau)
      \big]_\nu
  \big)
  =
  2 \sqrt{\pi}
  \mfrac{\Ga(2\nu+\frac{1}{2})}{\Ga(2\nu+1)}\,
  \mbinom{2\nu}{\nu}
  \mbinom{2\nu-\frac{1}{2}}{\nu}^{-1}
  \td{n}^{\nu + \frac{1}{2}}\,
  e(\td{n}\tau)
  \tx{.}
\end{gather*}
We write the binomial coefficients in terms of the~$\Gamma$\nbd{}function and use the Legendre duplication formula to confirm the first formula.

Using the Legendre duplication formula again, we find that
\begin{gather*}
  \mbinom{\nu+\frac{1}{2}}{\nu-t}
  \mbinom{\nu-\frac{1}{2}}{t}
  =
  \mbinom{\nu-\frac{1}{2}}{\nu}
  \mbinom{2\nu+1}{2t+1}
  \tx{.}
\end{gather*}
Inserting this into \fref{Proposition}{prop:holomorphic_projection:fourier_term_rankin_cohen_nonconstant} gives
\begin{gather*}
  \pi^\hol_{2+2\nu}\big( \big[
      \Ga\big(-\tfrac{1}{2}, 4 \pi |n| y\big)
      e(n \tau),\,
      e(\td{n} \tau)
      \big]_\nu
  \big)
  =
  \Ga\big( -\tfrac{1}{2} \big)
  \sum_{\substack{t \in \frac{1}{2} + \ZZ \tx{ or} \\ t \in \{0,\ldots,\nu\}}}
  (-1)^{2t}
  \mbinom{\nu-\frac{1}{2}}{\nu}
  \mbinom{2\nu+1}{2t+1}\,
  |n|^t \td{n}^{\nu-t}\,
  e\big( (n+\td{n})\tau \big)
  \tx{.}
\end{gather*}
For ease of notation, we set~$u = 2t + 1$. Note that the second binomial coefficient in the sum vanishes if~$u < 0$ or~$u > 2 \nu + 1$. Hence,
\begin{gather*}
  2\sqrt{\pi}
  \mbinom{\nu-\frac{1}{2}}{\nu}
  \sum_{u=0}^{2\nu+1}
  (-1)^u
  \mbinom{2\nu+1}{u}\,
  \Big(
  \mfrac{|n|}{\td{n}}
  \Big)
  ^{\frac{u-1}{2}}
  \td{n}^\nu
  =
  2\sqrt{\pi}
  \mbinom{\nu-\frac{1}{2}}{\nu}\,
  \Big(
  1
  -
  \Big(
  \mfrac{|n|}{\td{n}}
  \Big)^{\frac{1}{2}}
  \Big)^{2\nu+1}
  \Big(
  \mfrac{|n|}{\td{n}}
  \Big)^{-\frac{1}{2}}
  \td{n}^{\nu}
\end{gather*}
implies the claim.
\end{proof}

We now combine the previous Lemmas and Propositions to obtain expressions for the holomorphic projection of~$\big[ E^{\rmv\rmv\,+}_{\frac{3}{2}}, \theta^{\rmv\rmv} \big]_\nu^\otimes$ and~$\big[ E^{\rmv\rmv\,-}_{\frac{3}{2}}, \theta^{\rmv\rmv} \big]_\nu^\otimes$.
Recall the order of components of the vector-valued modular form for~$\ov\rhoH \otimes \rhoH$ that we specified in \fref{Section}{ssec:vector_valued_modular_forms:tensor_products}, when defining the tensor product.

\begin{proposition}
\label{prop:holomorphic_projection:hurwitz:holomorphic_part}
For an integer~$\nu \ge 0$,
\begin{gather*}
  \pi^\hol_{2 + 2\nu}\big(
  \big[ E^{\rmv\rmv\,+}_{\frac{3}{2}}, \theta^{\rmv\rmv} \big]_\nu^\otimes
  \big)
  =
  \big[ E^{\rmv\rmv\,+}_{\frac{3}{2}}, \theta^{\rmv\rmv} \big]_\nu^\otimes
  =
  \Big(
  \sum_{N = 0}^\infty
  c^+_i(N)\,
  e
  \big(
  \tfrac{N}{4}
  \tau
  \big)
  \Big)
  _
  {1\leq i\leq 4}
  \tx{,}
\end{gather*}
with coefficients
\begin{alignat*}{2}
  c^+_1(N)
   & =
  \sum_{\mu=0}^\nu
  r_\mu\;
  \sum_{m^2\leq N \slash 4}
  \mspace{-6mu}
  H(N-4m^2)\,
  \big(N- 4 m^2 \big)^\mu\,
  \big(4 m^2 \big)^{\nu-\mu}
  \tx{,}\qquad
   &   &
  \tx{if\/ } N \equiv 0 \;\pmod{4}
  \tx{;}
  \\
  c^+_2(N)
   & =
  \sum_{\mu=0}^\nu
  r_\mu\;
  \sum_{m^2\leq N}
  H(N-m^2)\,
  \big(
  N-m^2
  \big)^\mu\,
  \big(
  m^2
  \big)
  ^{\nu-\mu}
  \tx{,}\qquad
   &   &
  \tx{if\/ } N \equiv 1 \;\pmod{4}
  \tx{;}
  \\
  c^+_3(N)
   & =
  \sum_{\mu=0}^\nu
  r_\mu\;
  \sum_{m^2\leq N}
  H(N-m^2)\,
  \big(
  N-m^2
  \big)^\mu\,
  \big(
  m^2
  \big)
  ^{\nu-\mu}
  \tx{,}\qquad
   &   &
  \tx{if\/ } N \equiv 3 \;\pmod{4}
  \tx{;}
  \\
  c^+_4(N)
   & =
  \mathmakebox[0pt][l]{
    \sum_{\mu=0}^\nu
    r_\mu\;
    \sum_{m^2+m+1\leq N \slash 4}
    \mspace{-18mu}
    H(N-4m^2-4m-1)\,
    \big(
    N-4m^2-4m-1
    \big)^\mu\,
    \big(
    4m^2+4m+1
    \big)^{\nu-\mu}
    \tx{,}
  }
  \\
   &   &   &
  \tx{if\/ } N \equiv 0 \;\pmod{4}
  \tx{;}
  \mspace{100mu} %
\end{alignat*}
and~$c^+_i(N) = 0$ in the remaining cases, where
\begin{gather*}
  r_\mu
  =
  \mfrac{(-1)^\mu}{4^\nu}\,
  \mbinom{\nu+\frac{1}{2}}{\nu-\mu}
  \mbinom{\nu-\frac{1}{2}}{\mu}
  \tx{.}
\end{gather*}
\end{proposition}

\begin{proof}
We prove the formula for~$c^+_4(N)$, the most intricate one, to illustrate the general procedure. Computing the derivative, we find that for
\begin{alignat*}{3}
  f
        & =
  \sum_{n \in 3 + 4\ZZ}
  \mspace{-6mu}
  H(n)\,
  e \big( \tfrac{n}{4}\tau \big)
  \quad &   & \tx{and}\quad
        &
  g
        & =
  \sum_{m \in 1 + 2\ZZ}
  \mspace{-6mu}
  e \big( \tfrac{m^2}{4}\, \tau \big)
  \tx{,}
  \intertext{
    we have for all~$\mu>0$:
  }
  f^{(\mu)}
        & =
  \sum_{n \in 3 + 4\ZZ}
  \mspace{-6mu}
  H(n)\,
  \big(
  \mfrac{n}{4}
  \big)^\mu\,
  e
  \big(
  \tfrac{n}{4}
  \tau
  \big)
  \quad &   & \tx{and}\quad
        &
  g^{(\mu)}
        & =
  \sum_{m \in 1 + 2\ZZ}
  \mspace{-6mu}
  \big(
  \mfrac{m^2}{4}
  \big)^{\mu}\,
  e
  \big(
  \tfrac{m^2}{4}
  \tau
  \big)
  \tx{.}
\end{alignat*}
The definition of the Rankin--Cohen brackets in~\eqref{eq:def:rankin-cohen} then yields
\begin{gather*}
  [f,g]_\nu
  =
  \sum_{\mu=0}^\nu
  (-1)^\mu\,
  \mbinom{\nu+\frac{1}{2}}{\nu-\mu}
  \mbinom{\nu-\frac{1}{2}}{\mu}
  \sum_{\substack{
      n \in 3 + 4\ZZ \\
      m \in 1 + 2\ZZ
    }}
  \mspace{-6mu}
  H(n)\,
  \big(
  \mfrac{n}{4}
  \big)^\mu\,
  \big(
  \mfrac{m^2}{4}
  \big)^{\nu-\mu}\,
  e\big(
  \big(
    \tfrac{n}{4}
    +
    \tfrac{m^2}{4}
    \big)
  \tau
  \big)
  \tx{.}
\end{gather*}
Grouping terms of the same exponents together, while moving the power of~$4$ to the front, gives
\begin{gather}
  \label{eq:la:holomorphic_projection:hurwitz:holomorphic_part}
  c^+_4(N)
  =
  \sum_{\mu=0}^\nu
  \mfrac{(-1)^\mu}{4^\nu}\,
  \mbinom{\nu+\frac{1}{2}}{\nu-\mu}
  \mbinom{\nu-\frac{1}{2}}{\mu}
  \sum_{\substack{
      N=n+m^2 \\
      n \in 3 + 4\ZZ \\
      m \in 1 + 2\ZZ
    }}
  \mspace{-9mu}
  H(n)\,
  n^\mu\,
  (m^2)^{\nu-\mu}
  \tx{.}
\end{gather}
We now change the indices of the inner sum to be integers. Making the substitutions~$n \leftarrow 4n+3$ and~$m \leftarrow 2m+1$ and then eliminating~$n$ yields
\begin{align*}
  c^+_4(N)
   & =
  \sum_{\mu=0}^\nu
  r_\mu
  \sum_{
    \substack{
      N=4(n+m^2+m+1)
  \\
      m\in\ZZ,n\in\ZZ_{\ge 0}
    }}
  \mspace{-24mu}
  H(4n+3)\,
  ( 4 n+3 )^\mu\,
  ( 4 m^2+4m+1 )^{\nu-\mu}
  \\
   & =
  \sum_{\mu=0}^\nu
  r_\mu
  \sum_{m^2+m+1\leq N \slash 4}
  \mspace{-24mu}
  H( N-4m^2-4m-1)\,
  \big( N-4m^2-4m-1 \big)^\mu\,
  \big( 4 m^2+4m+1 \big)^{\nu-\mu}
  \tx{.}
\end{align*}
The identities for~$c^+_1(N)$,~$c^+_2(N)$, and~$c^+_3(N)$ follow analogously.
\end{proof}

\begin{proposition}
\label{prop:holomorphic_projection:hurwitz:nonholomorphic_part}
We have
\begin{gather*}
  \pi^\hol_{2 + 2\nu}\big( \big[
    E^{\rmv\rmv\,-}_{\frac{3}{2}},\, \theta^{\rmv\rmv}
    \big]^\otimes_\nu \big)
  =
  \Big(
  \sum_{N = 0}^\infty
  c^-_i(N)\,
  e
  \big(
  \tfrac{N}{4}
  \tau
  \big)
  \Big)
  _
  {1\leq i\leq 4}
  \tx{,}
\end{gather*}
with
\begin{alignat*}{2}
  c^-_1(N)
   & =
  \mbinom{\nu-\frac{1}{2}}{\nu} \,
  \lambda_{2\nu+1}^{\mathrm{even}} \big( \mfrac{N}{4} \big)
  \tx{,}\qquad
   &   &
  \tx{if\/ }
  N \equiv 0 \;\pmod{4}
  \tx{;}
  \\
  c^-_2(N)
   & =
  \mbinom{\nu-\frac{1}{2}}{\nu} \,
  \mfrac{1}{2^{2\nu+1}}\,
  \lambda_{2\nu+1}(N)
  \tx{,}\qquad
   &   &
  \tx{if\/ }
  N \equiv 1 \;\pmod{4}
  \tx{;}
  \\
  c^-_3(N)
   & =
  \mbinom{\nu-\frac{1}{2}}{\nu} \,
  \mfrac{1}{2^{2\nu+1}}\,
  \lambda_{2\nu+1}(N)
  \tx{,}\qquad
   &   &
  \tx{if\/ }
  N \equiv 3 \;\pmod{4}
  \tx{;}
  \\
  c^-_4(N)
   & =
  \mbinom{\nu-\frac{1}{2}}{\nu} \,
  \lambda_{2\nu+1}^{\mathrm{odd}} \big( \mfrac{N}{4} \big)
  \tx{,}\qquad
   &   &
  \tx{if\/ }
  N \equiv 0 \;\pmod{4}
  \tx{;}
\end{alignat*}
and~$c^-_i(N) = 0$ in the remaining cases, where~$\lambda_k^{\mathrm{even}}(N)$ and~$\lambda_k^{\mathrm{odd}}(N)$ are as in~\eqref{eq:holomorphic_projection:hurwitz:lambdas}.
\end{proposition}

\begin{proof}
We give full details for the calculation of~$c^-_1(N)$, and at the end of the proof indicate the required adjustments to determine the remaining~$c^-_i(N)$. First, we evaluate
\begin{align*}
      &
  \big( \big[ E^{\rmv\rmv\,-}_{\frac{3}{2}} , \theta^{\rmv\rmv} \big]^\otimes_\nu \big)_1
  =
  \Bigg[
    \mfrac{1}{4 \pi} y^{-\frac{1}{2}}
    +
    \mfrac{1}{8 \sqrt{\pi}}
    \sum_{m \in 2\ZZ \setminus \{0\}}
    |m|\,
    \Gamma\big( -\tfrac{1}{2}, \pi m^2 y \big)\,
    e\big( -\tfrac{m^2}{4} \tau \big)
    ,\;
    \sum_{\td{m} \in 2\ZZ}
    e\big( \tfrac{\td{m}^2}{4} \tau \big)
    \Bigg]_\nu
  \\
  ={} &
  \Big[
    \mfrac{1}{4 \pi} y^{-\frac{1}{2}}
    ,\,
    1
    \Big]_\nu
  \pluspad
  \sum_{\td{m} \in 2\ZZ \setminus\{0\}}
  \Big[
    \mfrac{1}{4 \pi} y^{-\frac{1}{2}}
    ,\,
    e\big( \tfrac{\td{m}^2}{4} \tau \big)
    \Big]_\nu
  +
  \sum_{\substack{m \in 2\ZZ \setminus \{0\} \\ \td{m} \in 2\ZZ}}
  \Big[
    \mfrac{1}{8 \sqrt{\pi}}
    |m|\,
    \Gamma\big( -\tfrac{1}{2}, \pi m^2 y \big)\,
    e\big( -\tfrac{m^2}{4} \tau \big)
    ,\,
    e\big( \tfrac{\td{m}^2}{4} \tau \big)
    \Big]_\nu
  \tx{.}
\end{align*}
The first term vanishes under~$\pi^\hol_{2+2\nu}$ by \fref{Proposition}{prop:holomorphic_projection:fourier_term_rankin_cohen_zeroth},
and the latter two terms are determined in \fref{Lemma}{la:holomorphic_projection_fourier_coefficients:hurwitz_class_numbers:simplification}.
More precisely, the second term becomes
\begin{gather*}
  \sum_{\td{m} \in 2\ZZ\setminus \{0\}}
  \mfrac{1}{2}
  \mbinom{\nu-\frac{1}{2}}{\nu}\,
  \Big( \mfrac{|\td{m}|}{2} \Big)^{2\nu+1}\,
  e\big(
  \tfrac{\td{m}^2}{4}
  \tau
  \big)
  =
  \mbinom{\nu-\frac{1}{2}}{\nu}\,
  \sum_{\td{m} = 1}^\infty
  {\td{m}}^{2\nu+1}
  e(\td{m}^2 \tau)
\end{gather*}
after taking~$\td{m} \ra 2\td{m}$ and noting that a given~$|\td{m}| > 0$ occurs twice as~$\td{m}$ ranges over~$\ZZ \setminus \{0\}$.
To handle the third term, we first observe that~$\pi^\hol_{2+2\nu}$ annihilates by definition contributions of negative exponents, i.e., the contributions with~$|\td{m}| < |m|$, and also annihilates the contributions of exponent zero, i.e., the contributions with~$|\td{m}| = |m|$, which decay exponentially.
The third term becomes
\begin{align*}
  \mfrac{1}{2}
  \mbinom{\nu-\frac{1}{2}}{\nu}
  \sum_{\substack{m,\td{m} \in 2\ZZ \setminus \{0\} \\ |\td{m}| > |m|}}
  \Big(
  \mfrac{|\td{m}|}{2}
  -
  \mfrac{|m|}{2}
  \Big)^{2\nu + 1}\,
  e
  \big(
  \big(
    \tfrac{\td{m}^2}{4}
    -
    \tfrac{m^2}{4}
    \big)
  \tau
  \big)
   & =
  2
  \mbinom{\nu-\mfrac{1}{2}}{\nu}
  \sum_{\substack{m,\td{m} \in \ZZ_{> 0}            \\ \td{m} > m}}
  (\td{m}-m)^{2\nu+1}
  e
  \big(
  (\td{m}^2-m^2)
  \tau
  \big)
\end{align*}
by similar reasoning to before. Combining the second and third term, the coefficient of~$e(\frac{N}{4} \tau)$ is then
\begin{align*}
  \mbinom{\nu-\frac{1}{2}}{\nu}\,
  \mspace{-8mu}
  \sum_{\substack{ \td{m} > 0     \\ \td{m}^2 = \frac{N}{4}}}
  \td{m}^{2\nu+1}
  \,+\,
  2
  \mbinom{\nu-\frac{1}{2}}{\nu}\,
  \mspace{-32mu}
  \sum_{\substack{ \td{m} > m > 0 \\ (\td{m}-m)(\td{m}+m) = \frac{N}{4}}}
  \mspace{-12mu}
  (\td{m}-m)^{2\nu+1}
   & =
  \mbinom{\nu-\frac{1}{2}}{\nu}\,
  \lambda_{2\nu+1}^{\mathrm{even}}(\tfrac{N}{4})
  \tx{.}
\end{align*}

In the calculation of~$c^-_4(N)$, we consider the sum over odd~$m$ and~$\td{m}$, in which case the exponent~$\frac{N}{4} = \frac{\td{m}^2}{4}-\frac{m^2}{4}$ is still an integer, but the divisors of
\begin{gather*}
  \mfrac{N}{4}
  =
  \big(
  \mfrac{\td{m}-m}{2}
  \big)
  \big(
  \mfrac{\td{m}+m}{2}
  \big)
\end{gather*}
have different parity. Finally, for~$c^-_2(N)$ and~$c^-_3(N)$ observe that~$\td{m}^2-m^2$ are congruent to~$1$ or~$3$ modulo~$4$, respectively.
\end{proof}

\section{Proof of the main theorem}%
\label{sec:proof_main_theorem}

\begin{proof}[Proof of \fref{Theorem}{mainthm:hurwitz_class_number_recursions}]
For~$\nu\ge 1$, observe that
\begin{align*}
  \pi^\hol_{2+2\nu} \big(
  \big[
    E^{\rmv\rmv}_{\frac{3}{2}},
    \theta^{\rmv\rmv}
    \big]_\nu^\otimes
  \big)
  & \;\in\;
  \rmS_{2+2\nu}\bigl( \ov\rhoH \otimes \rhoH \bigr)
  \quad\tx{and}
  \\
  \pi^\hol_{2+2\nu} \big(
  \big[
    E^{\rmv\rmv}_{\frac{3}{2}},
    \theta^{\rmv\rmv}
    \big]_\nu^\otimes
  \big)
  &=
  \pi^\hol_{2+2\nu} \big(
  \big[
    E^{\rmv\rmv +}_{\frac{3}{2}},
    \theta^{\rmv\rmv}
    \big]_\nu^\otimes
  \big)
  +
  \pi^\hol_{2+2\nu} \big(
  \big[
    E^{\rmv\rmv -}_{\frac{3}{2}},
    \theta^{\rmv\rmv}
    \big]_\nu^\otimes
  \big)
  \tx{.}
\end{align*}
The formulas in \fref{Theorem}{mainthm:hurwitz_class_number_recursions} then follow from \fref{Corollary}{cor:representation_theory:hurwitz_modular_forms_support}, \fref{Proposition}{prop:holomorphic_projection:hurwitz:holomorphic_part}, and \fref{Proposition}{prop:holomorphic_projection:hurwitz:nonholomorphic_part}.
More specifically, \hyperref[it:cor:representation_theory:hurwitz_modular_forms_support:weight4]{Statement~\ref*{it:cor:representation_theory:hurwitz_modular_forms_support:weight4}}
of \fref{Corollary}{cor:representation_theory:hurwitz_modular_forms_support} shows that if~$\nu = 1$, then
\begin{gather*}
  c_1^+(4N) + c_1^-(4N) = 0
  \quad\tx{and}\quad
  c_4^+(4N) + c_4^-(4N) = 0
  \tx{,}
\end{gather*}
which will yield the first and second recursion in \fref{Theorem}{mainthm:hurwitz_class_number_recursions}, respectively.
\hyperref[it:cor:representation_theory:hurwitz_modular_forms_support:weight6]{Statement~\ref*{it:cor:representation_theory:hurwitz_modular_forms_support:weight6}} of the corollary shows that if~$\nu = 2$, then
\begin{gather*}
  c_1^+(8N) + c_1^-(8N) = 0
  \quad\tx{and}\quad
  c_4^+(8N) + c_4^-(8N) = 0
  \tx{,}
\end{gather*}
which will yield the third and fourth. We have to evaluate~$c_1^\pm$ and~$c_4^\pm$ to prove the theorem.

The coefficients~$c_1^-$ (even case) and~$c_4^-$ (odd case) in \fref{Proposition}{prop:holomorphic_projection:hurwitz:nonholomorphic_part} equal the right hand sides of the recursion in \fref{Theorem}{mainthm:hurwitz_class_number_recursions} up to rescaling by a binomial factor. The polynomials~$g^{\mathrm{even}/\mathrm{odd}}_{2\nu + 1}$ in \fref{Theorem}{mainthm:hurwitz_class_number_recursions} arise as the negative of the product of the inverse of these binomial factors with the coefficients of the Hurwitz class numbers in the expressions for~$c_1^+$~(even case) and~$c_4^+$~(odd case) given in \fref{Proposition}{prop:holomorphic_projection:hurwitz:holomorphic_part}, after replacing~$N$ by~$4 N$.
We obtain
\begin{align*}
  g^{\mathrm{even}}_{2\nu+1}(N, m)
   & =
  -\mbinom{\nu-\frac{1}{2}}{\nu}^{-1}\,
  \sum_{\mu=0}^\nu
  r_\mu\,
  \big(
  4N - 4m^2
  \big)^\mu\,
  (4m^2)^{\nu-\mu}
  \tx{,}
  \\
  g^{\mathrm{odd}}_{2\nu+1}(N, m)
   & =
  -\mbinom{\nu-\frac{1}{2}}{\nu}^{-1}\,
  \sum_{\mu=0}^\nu
  r_\mu\,
  \big(
  4N - 4m^2 - 4m - 1
  \big)^\mu\,
  \big(
  4m^2 + 4m + 1
  \big)^{\nu-\mu}
  \tx{,}
\end{align*}
which after simplification yields the expressions in \fref{Theorem}{mainthm:hurwitz_class_number_recursions}.
\end{proof}

To conclude, we show how the relations of~\cite{cohen-1975} and~\cite{mertens-2014} arise from our framework.
Observe that, by Equation~\eqref{eq:la:holomorphic_projection:hurwitz:holomorphic_part},
\begin{gather*}
  c_4^+(4N)
  =
  \sum_{\mu=0}^\nu
  r_\mu
  \sum_{\substack{
      4N=n+m^2 \\
      n \in 3 + 4\ZZ \\
      m \in 1 + 2\ZZ
    }}
  \mspace{-12mu}
  H(n)\,
  n^\mu\,
  (m^2)^{\nu-\mu}
  \tx{,}
\end{gather*}
and similarly, one can show
\begin{gather*}
  c_1^+(4N)
  =
  \sum_{\mu=0}^\nu
  r_\mu
  \sum_{\substack{
      4N=n+m^2 \\
      n \in 4\ZZ \\
      m \in 2\ZZ
    }}
  \mspace{-12mu}
  H(n)\,
  n^\mu\,
  (m^2)^{\nu-\mu}
  \tx{.}
\end{gather*}
If~$4N=n+m^2$ is an integer, then either~$n \equiv -1 \,\pmod{4}$ and~$m$ is odd, or~$n \equiv 0 \,\pmod{4}$ and~$m$ is even.
Hence, for fixed~$N$, the two sums together run over all possibilities for~$n$ and~$m$, yielding
\begin{gather*}
  c_1^+(4N)+c_4^+(4N)
  =
  \sum_{\mu=0}^\nu
  r_\mu
  \sum_{4N=n+m^2}
  H(n)
  n^\mu
  (m^2)^{\nu-\mu}
  =
  \sum_{\mu=0}^\nu
  r_\mu
  \sum_{m^2\le 4N}
  H(4N-m^2)
  \big(
  4N-m^2
  \big)^\mu
  \big(
  m^2
  \big)^{\nu-\mu}
  \tx{.}
\end{gather*}
At the same time, we have
\begin{gather*}
  c_1^-(4N)+c_4^-(4N)
  =
  \mbinom{\nu-\frac{1}{2}}{\nu}
  \lambda_{2\nu+1}(N)
  \tx{.}
\end{gather*}
\hyperref[it:cor:representation_theory:hurwitz_modular_forms_support:weight6to10]{Statement~\ref*{it:cor:representation_theory:hurwitz_modular_forms_support:weight6to10}} of \fref{Corollary}{cor:representation_theory:hurwitz_modular_forms_support} gives
\begin{gather*}
  \big(
  c_1^+(4N)+c_4^+(4N)
  \big)
  +
  \big(
  c_1^-(4N)+c_4^-(4N)
  \big)
  =0
  \tx{.}
\end{gather*}
Thus,
\begin{gather*}
  \sum_{m^2\le 4N}
  g_{2\nu+1}(N,m)
  H(4N-m^2)
  +
  \lambda_{2\nu+1}(N)
  =0
\end{gather*}
for~$\nu\in\{1,2,3,4,6\}$, where
\begin{gather*}
  g_{2\nu+1}(N,m)
  =
  \mbinom{\nu-\frac{1}{2}}{\nu}^{-1}
  \sum_{\mu=0}^\nu
  r_\mu
  (4N-m^2)^\mu
  (m^2)^{\nu-\mu}
  \tx{.}
\end{gather*}
Evaluating this expression for the specified values of~$\nu$ yields the family of relations in~\cite{cohen-1975}.

Furthermore, observe~$N$ is odd if and only if~$N \equiv 1 \,\pmod{4}$ or~$N \equiv 3 \,\pmod{4}$. \fref{Proposition}{prop:holomorphic_projection:hurwitz:holomorphic_part} yields
\begin{gather*}
  c_2^+(N)+c_3^+(N)
  =
  \sum_{\mu=0}^\nu
  r_\mu
  \sum_{m^2\le N}
  H(N-m^2)
  \big(
  N-m^2
  \big)^\mu
  \big(
  m^2
  \big)^{\nu-\mu}
\end{gather*}
for all odd~$N$.
Likewise, \fref{Proposition}{prop:holomorphic_projection:hurwitz:nonholomorphic_part} yields
\begin{gather*}
  c_2^-(N)+c_3^-(N)
  =
  \mbinom{\nu-\mfrac{1}{2}}{\nu}
  \mfrac{1}{2^{2\nu+1}}
  \lambda_{2\nu+1}(N)
\end{gather*}
for odd~$N$.
\hyperref[it:cor:representation_theory:hurwitz_modular_forms_support:weight4]{Statement~\ref*{it:cor:representation_theory:hurwitz_modular_forms_support:weight4}} of \fref{Corollary}{cor:representation_theory:hurwitz_modular_forms_support} shows that when~$\nu=1$,
\begin{gather*}
  \big(
  c_2^+(N)+c_3^+(N)
  \big)
  +
  \big(
  c_2^-(N)+c_3^-(N)
  \big)
  =0
  \tx{.}
\end{gather*}
Evaluating as in \fref{Theorem}{mainthm:hurwitz_class_number_recursions} yields the formula from~\cite{mertens-2014}.
Note that, unlike in \fref{Theorem}{mainthm:hurwitz_class_number_recursions}, looking at the components individually does not help, as the second and third components have disjoint support.

\subsection*{Funding}
M.R. was partially supported by Vetenskapsr\aa det Grant~2023-04217.
O.K.R. was partially supported by a grant from the Simons Foundation (\#835652 to Olav Richter).

\subsection*{Data availability}
All data used are included in the manuscript.

\subsection*{Conflict of interest}
The authors have no conflict of interest to declare.

\vspace{1.5\baselineskip}
\phantomsection
\addcontentsline{toc}{section}{References}
\markright{References}
\label{sec:references}
{
  \sloppy
  \linespread{0.8}
  \printbibliography[heading=none]
}

\filbreak
\Needspace*{5\baselineskip}
\noindent%
\rule{\textwidth}{0.15em}
\\\nopagebreak

{\small\noindent
  Matthew Ortiz\\\nopagebreak
  Department of Mathematics\\\nopagebreak
  University of North Texas\\\nopagebreak
  Denton, TX 76203, USA\\\nopagebreak
  E-mail: \url{matthewortiz2@my.unt.edu}
}\vspace{.5\baselineskip}

{\small\noindent
  Martin Raum\\\nopagebreak
  Chalmers tekniska högskola och G\"oteborgs Universitet\\\nopagebreak
  Institutionen f\"or Matematiska vetenskaper\\\nopagebreak
  SE-412 96 G\"oteborg, Sweden\\\nopagebreak
  E-mail: \url{martin@raum-brothers.eu}%
}\vspace{.5\baselineskip}

{\small\noindent
  Olav K. Richter\\\nopagebreak
  Department of Mathematics\\\nopagebreak
  University of North Texas\\\nopagebreak
  Denton, TX 76203, USA\\\nopagebreak
  E-mail: \url{richter@unt.edu}%
}%

\ifdraft{%
  \listoftodos%
}

\end{document}

